\documentclass[12pt]{article}

\usepackage[margin=1in]{geometry}
\usepackage{amssymb}
\usepackage{amsmath}
\usepackage{amsthm}
\usepackage{graphicx}
\usepackage{url}
\usepackage{color}
\usepackage{framed}
\usepackage[table]{xcolor}
\usepackage[ruled,vlined]{algorithm2e}
\usepackage{mathrsfs}
\usepackage{tikz}
\usetikzlibrary{patterns}

\newtheorem{theorem}{Theorem}%[section]

\newtheorem{lemma}[theorem]{Lemma}
\newtheorem{proposition}[theorem]{Proposition}

\theoremstyle{definition}

\newtheorem{definition}[theorem]{Definition}

\title{Matching component analysis for transfer learning}

\author{Charles~Clum\footnote{Department of Mathematics, The Ohio State University, Columbus, OH~~(\texttt{clum.47@osu.edu})}\qquad Dustin~G.~Mixon\footnote{Department of Mathematics, The Ohio State University, Columbus, OH~~(\texttt{mixon.23@osu.edu})}\qquad Theresa~Scarnati\footnote{Air Force Research Laboratory, Wright-Patterson AFB, OH~~(\texttt{theresa.scarnati.1@us.af.mil})}}
\date{}

\begin{document}
\maketitle

\begin{abstract}
We introduce a new Procrustes-type method called matching component analysis to isolate components in data for transfer learning.
Our theoretical results describe the sample complexity of this method, and we demonstrate through numerical experiments that our approach is indeed well suited for transfer learning.
\end{abstract}

\section{Introduction}

Many state-of-the-art classification algorithms require a large training set that is statistically similar to the test set.
For example, deep learning--based approaches require a large number of representative samples in order to find near-optimal network weights and biases~\cite{higham2018deep, krizhevsky2012imagenet}.
Similarly, template-based approaches require large dictionaries of training images so that each test image can be represented by an element of the dictionary~\cite{tosic2011dictionary, yang2014sparse, paulson2018synthetic, irving1999classification}.
For each technique, if test images cannot be represented in a feature space that has been determined from the training set, then classification accuracy is poor. 

In applications such as synthetic aperture radar (SAR) automatic target recognition (ATR), it is infeasible to collect the volume of data necessary to naively train high-accuracy classification networks.
Additionally, due to varying operating conditions, the features measured in SAR imagery are different from those extracted from electro-optical (EO) imagery~\cite{LewisEtal:19}.
As such, off-the-shelf networks that have been pre-trained on the popular EO-based ImageNet~\cite{deng2009imagenet} or CIFAR-10~\cite{krizhevsky2010convolutional} datasets are insufficient for performing accurate ATR tasks in different imaging domains.
In fact, recent work has demonstrated that pre-trained networks fail to effectively generalize to random perturbations on test sets~\cite{recht2019imagenet, recht2018cifar}.
To build more representative training sets, additional data are often generated using modeling and simulation software.
However, due to various model errors, simulated data often misrepresent the real-world scattering observed in measured imagery.
Thus, even though it is possible to augment training sets with a large amount of simulated data, the inherent differences in sensor modalities and data representations make modifying classification networks a non-trivial task~\cite{scarnati2019deep}. 

In this paper, we introduce \textbf{matching component analysis (MCA)} to help remedy this situation.
Given a small number of images from the training domain and matching images from the testing domain, MCA identifies a low-dimensional feature space that both domains have in common.
With the help of MCA, one can map augmented training sets into a common domain, thereby making the classification task more robust to mismatch between the training and testing domains.
We note that other transfer learning methods, image-to-image domain regression techniques, and generative adversarial networks have all been developed with a similar task in mind~\cite{lewis2018generative, isola2017image, tzeng2017adversarial, motiian2017few, lee2018diverse}, but little theory has been developed to explain the performance of these machine learning--based adaptation techniques.
By contrast, in this paper, we estimate the number of matched samples needed for MCA to identify a common domain.

The rest of the paper is organized as follows.
Section~2 introduces the MCA algorithm and our main theoretical results.
In Section~3, we use a sequence of numerical experiments involving MNIST~\cite{LeCunCB:online} and SAR~\cite{LewisEtal:19} data to demonstrate that classifying data in the common domain allows for more accurate classification.
We discuss limitations of MCA in Section~4.
Sections~5 and~6 contain the proofs of our main theoretical results. 

\section{Matching component analysis}

Let $\mathbb{R}^{d_1}$ and $\mathbb{R}^{d_2}$ denote the training and testing domains, respectively.
Traditionally, our training set would consist of labeled points in $\mathbb{R}^{d_1}$, whereas our test test would consist of labeled points in $\mathbb{R}^{d_2}$.
%Suppose we receive a training set of $N$ labeled points in $\mathbb{R}^{d_1}$ and of $n\ll N$ labeled points in $\mathbb{R}^{d_2}$.
In order to bridge the disparity between the training and testing domains, we will augment our training set with a \textbf{matching set} of $n$ labeled pairs in $\mathbb{R}^{d_1}\times\mathbb{R}^{d_2}$.
Then our full training set, whose size we denote by $N\gg n$, consists of a \textbf{conventional training set} of $N-n$ labeled points in $\mathbb{R}^{d_1}$ and a matching set of $n$ labeled points in $\mathbb{R}^{d_1}\times \mathbb{R}^{d_2}$.
The matching set will enable us to identify maps $g_1$ and $g_2$ from the training and testing domains to a common domain $\mathbb{R}^k$, where we can train a classifier $h$ on the full training set:
\begin{center}
\begin{picture}(400,120)
\thicklines
\put(7,100){\framebox{training domain}}
\put(10,15){\framebox{testing domain}}
\put(103,90){\vector(2,-1){38}}
\put(120,90){$g_1$}
\put(100,30){\vector(2,1){40}}
\put(120,30){$g_2$}
\put(150,56){\framebox{common domain}}
\put(250,60){\vector(1,0){50}}
\put(270,65){$h$}
\put(310,56){\framebox{label domain}}
\end{picture}
\end{center}
We model our setting in terms of unknown random variables $X_1\in\mathbb{R}^{d_1}$, $X_2\in\mathbb{R}^{d_2}$, $Y\in\mathbb{R}$ over a common probability space $(\Omega,\mathcal{F},\mathbb{P})$.
In particular, suppose points $\{\omega_j\}_{j\in[N]}$ are drawn independently at random from $(\Omega,\mathcal{F},\mathbb{P})$, and we are given
\[
\{X_1(\omega_j)\}_{j\in[N]},
\qquad
\{X_2(\omega_j)\}_{j\in[n]},
\qquad
\{Y(\omega_j)\}_{j\in[N]}
\]
for some $n\ll N$ with the task of finding $f\colon\mathbb{R}^{d_2}\to \mathbb{R}$ such that $f(X_2)\approx Y$.
%old version:
%\begin{itemize}
%\item
%there is an unknown probability distribution $\mathscr{D}$ over an unknown space $\mathbb{R}^D$;
%\item
%there is a known set $S$ of labels or values and an unknown function $\ell\colon\mathbb{R}^D\to S$;
%\item
%for each $s,t\in\mathbb{N}$, there is a known class $\mathscr{F}(s,t)$ of functions that map $\mathbb{R}^s$ to $\mathbb{R}^t$;  
%\item
%for each $i\in\{1,2\}$, there is a known $d_i\in\mathbb{N}$ and an unknown $f_i\in\mathscr{F}(D,d_i)$;
%\item
%unknown vectors $\{U_j\}_{j\in[N]}$ are drawn independently from $\mathscr{D}$;
%\item
%we are given $\{\ell(U_j)\}_{j\in[N]}$ and noisy versions of $\{f_1(U_j)\}_{j\in[N]}$ and $\{f_2(U_j)\}_{j\in[n]}$ for some $n\ll N$; and
%\item
%our task is to find $\tilde\ell\colon\mathbb{R}^{d_2}\to S$ that (approximately) minimizes
%\[
%\mathop\mathbb{E}_{U\sim\mathscr{D}}L\big((\tilde\ell\circ f_2)(U),\ell(U)\big)
%\]
%for some loss function $L$.
%\end{itemize}
Our approach is summarized by the following:
\begin{itemize}
\item[(i)]
Select $k\in\mathbb{N}$ and a class $\mathscr{F}_i$ of functions that map $\mathbb{R}^{d_i}$ to $\mathbb{R}^k$ for each $i\in\{1,2\}$.
\item[(ii)]
Use $\{X_1(\omega_j)\}_{j\in[n]}$ and $\{X_2(\omega_j)\}_{j\in[n]}$ to (approximately) solve
\begin{align}
\label{eq.general program}&\text{minimize}
\quad
\mathbb{E}\|g_1(X_1)-g_2(X_2)\|^2\\
\nonumber&\text{subject to}
\quad
g_i\in\mathscr{F}_i,
\quad
\mathbb{E}g_i(X_i)=0,
\quad
\mathbb{E}g_i(X_i)g_i(X_i)^\top=I_k,
\quad
i\in\{1,2\}.
\end{align}
\item[(iii)]
Train $h\colon\mathbb{R}^k\to \mathbb{R}$ on $\{g_1(X_1(\omega_j))\}_{j\in[N]}$ and $\{Y(\omega_j)\}_{j\in[N]}$, and return $f:=h\circ g_2$.
\end{itemize}
For (i), we are principally interested in the case where $\mathscr{F}_i$ is the set of affine linear transformations from $\mathbb{R}^{d_i}$ to $\mathbb{R}^k$.
This choice of function class is nice because it locally approximates arbitrary differentiable functions while being amenable to theoretical analysis.
Considering the ubiquity of principal component analysis in modern data science, this choice promises to be useful in practice.
The constraints in program~\eqref{eq.general program} ensure that the training set in (iii) is normalized, while simultaneously preventing useless choices for $g_i$, such as those for which $g_i(X_i)=0$ almost surely.
Intuitively, (ii) selects $g_1$ and $g_2$ so as to transform $X_1$ and $X_2$ into a common domain, and then (iii) leverages the large number of realizations of $X_1$ to predict $Y$ in this domain, thereby enabling us to predict $Y$ from $X_2$.
We expect this approach to work well in settings for which
\begin{itemize}
\item
each $g_i(X_i)$ captures sufficient information about $\omega$ to predict $Y$,
\item
$h$ is robust to slight perturbations so that $h(g_1(X_1))\approx h(g_2(X_2))$,
\item
$Y|X_2$ is too complicated to be learned from a training set of size $n$, and
\item
$Y|g_1(X_1)$ can be learned from a training set of size $N$.
\end{itemize}

To solve program~\eqref{eq.general program} in the case of affine linear transformations, $g_i$ must have the form $g_i(x)=A_ix+b_i$ for some $A_i\in\mathbb{R}^{k\times d_i}$ and $b_i\in\mathbb{R}^{k}$.
Let $\mu_i$ and $\Sigma_i$ denote the mean and covariance of $X_i$.
The constraint in program~\eqref{eq.general program} forces $A_i\mu_i+b_i=\mathbb{E}g_i(X_i)=0$, and so $b_i=-A_i\mu_i$, i.e., $g_i(x)=A_i(x-\mu_i)$.
The constraint also forces $A_i\Sigma_iA_i^\top=\mathbb{E}g_i(X_i)g_i(X_i)^\top=I_k$.
Overall, program~\eqref{eq.general program} is equivalent to
\begin{equation}
\label{eq.mca program 1}
\text{minimize}
\quad
\mathbb{E}\|A_1(X_1-\mu_1)-A_2(X_2-\mu_2)\|^2
\quad
\text{subject to}
\quad
A_i\Sigma_iA_i^\top=I_k,
\quad
i\in\{1,2\}.
\end{equation}
Notice that this program is not infeasible when $k>\min\{\operatorname{rank}\Sigma_1,\operatorname{rank}\Sigma_2\}$.
Of course, we do not have access to $X_1$ and $X_2$, but rather $n$ realizations of each, and so we are forced to approximate.
To this end, we estimate the means and covariances as
\begin{equation}
\label{eq.mean cov est}
\hat\mu_i
:=\frac{1}{n}\sum_{j\in[n]}X_i(\omega_j),
\qquad
\hat\Sigma_i
:=\frac{1}{n}\sum_{j\in[n]}(X_i(\omega_j)-\hat\mu_i)(X_i(\omega_j)-\hat\mu_i)^\top,
\end{equation}
and then consider the following approximation to program~\eqref{eq.mca program 1}:
\begin{align}
\label{eq.mca program 2}&\text{minimize}
\quad
\frac{1}{n}\sum_{j\in[n]}\|A_1(X_1(\omega_j)-\hat\mu_1)-A_2(X_2(\omega_j)-\hat\mu_2)\|^2\\
\nonumber&\text{subject to}
\quad
A_i\hat\Sigma_iA_i^\top=I_k,
\quad
i\in\{1,2\}.
\end{align}
Observe that program~\eqref{eq.mca program 2} is equivalent to
\begin{align}
\label{eq.mca program 3}&\text{minimize}
\quad
\frac{1}{n}\sum_{j\in[n]}\|A_1(X_1(\omega_j)-\hat\mu_1)-A_2(X_2(\omega_j)-\hat\mu_2)\|^2\\
\nonumber&\text{subject to}
\quad
A_i\hat\Sigma_iA_i^\top=I_k,
\quad
\operatorname{im}A_i^\top\subseteq\operatorname{im}\hat\Sigma_i,
\quad
i\in\{1,2\}.
\end{align}
Indeed, if $(A_1,A_2)$ is feasible in~\eqref{eq.mca program 2}, then we can project the rows of $A_i$ onto $\operatorname{im}\hat\Sigma_i$ without changing the objective value.
Next, define $r_i:=\operatorname{rank}\hat\Sigma_i$, take $V_i$ to be any $d_i\times r_i$ matrix whose columns form an orthonormal basis for $\operatorname{im}\hat\Sigma_i$, and define $Z_i$ to be the $r_i\times n$ matrix whose $j$th column is $V_i^\top(\hat\Sigma_i^\dagger)^{1/2}(X_i(\omega_j)-\hat\mu_i)$.
Then every solution of
\begin{equation}
\label{eq.mca program 4}
\text{minimize}
\quad
\frac{1}{n}\|B_1Z_1-B_2Z_2\|_F^2
\quad
\text{subject to}
\quad
B_iB_i^\top=I_k,
\quad
i\in\{1,2\}
\end{equation}
can be transformed to a solution to program~\eqref{eq.mca program 3} by the change of variables $A_i=B_iV_i^\top(\hat\Sigma_i^\dagger)^{1/2}$, where $B_i\in\mathbb{R}^{k\times r_i}$.
In fact, by this change of variables, programs~\eqref{eq.mca program 3} and~\eqref{eq.mca program 4} are equivalent.
In the special case where $k=d_1=d_2$, one may take $B_2=I_k$ without loss of generality, and then program~\eqref{eq.mca program 4} amounts to the well-known \textit{orthogonal Procrustes problem}~\cite{HornJ:12}.
In general, we refer to~\eqref{eq.mca program 4} as the \textbf{projection Procrustes problem}; see Figure~\ref{fig.mickey} for an illustration.
Considering orthogonal Procrustes enjoys a spectral solution, there is little surprise that projection Procrustes also enjoys a spectral solution:

\begin{figure}[t]
\centering
\includegraphics[width=\textwidth]{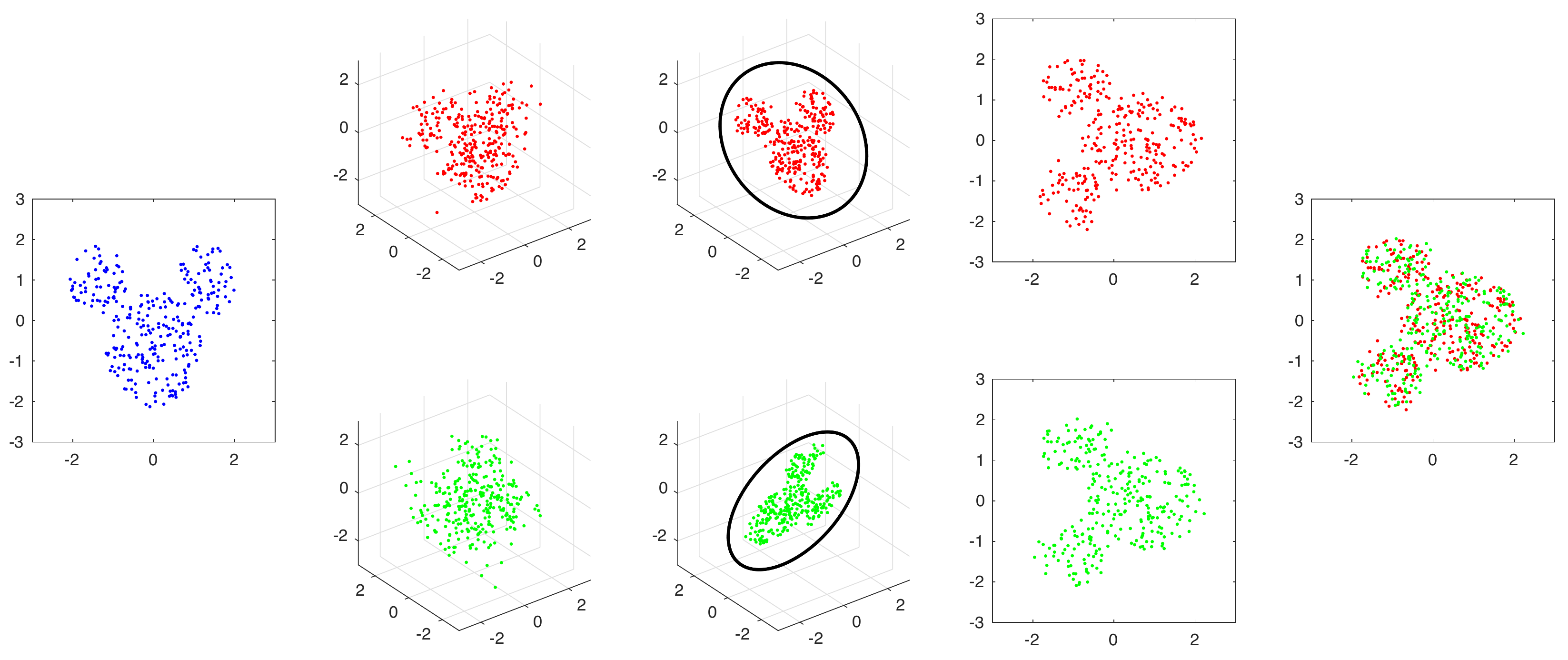}
\caption{\footnotesize{Illustration of the projection Procrustes problem. \textbf{(left)} Draw $300$ points from a uniform distribution over a Mickey Mouse shape in the $xy$-plane of $\mathbb{R}^3$. \textbf{(middle left)} Perform the following deformation twice in order to produce matched datasets $Z_1$ and $Z_2$: Add independent spherical Gaussian noise ($\sigma=0.1$) to each data point, randomly rotate the entire dataset, and then normalize the result to have zero mean and identity covariance. \textbf{(middle)} Solve the projection Procrustes problem for $Z_1$ and $Z_2$ with $k=2$. The optimal $B_1$ and $B_2$ have the property that $B_i^\top B_i$ is a $3\times 3$ orthogonal projection matrix of rank~$2$, and we plot the projected data $B_i^\top B_i Z_i$. \textbf{(middle right)} The resulting $2$-dimensional transformation of the data, namely, the columns of $B_iZ_i$. \textbf{(right)} We superimpose both datasets in the $2$-dimensional transform space to illustrate how well they are aligned.}\label{fig.mickey}}
\end{figure}

\begin{lemma}
\label{lem.proj proc prob}
Suppose $Z_iZ_i^\top=nI_{r_i}$ for both $i\in\{1,2\}$.
If $k>\min\{r_1,r_2\}$, then the projection Procrustes problem~\eqref{eq.mca program 4} is infeasible.
Otherwise, select any $k$-truncated singular value decomposition $W_1\Sigma W_2^\top$ of $Z_1Z_2^\top$.
Then $(B_1,B_2)=(W_1^\top,W_2^\top)$ is a solution to~\eqref{eq.mca program 4}.
\end{lemma}

\begin{proof}
Since $B_i$ is a $k\times r_i$ matrix, the constraint $B_iB_i^\top=I_k$ requires $k\leq r_i$.
Suppose $k\leq\min\{r_1,r_2\}$, and consider any feasible point $(B_1,B_2)$ in program~\eqref{eq.mca program 4}.
Then
\[
\|B_iZ_i\|_F^2
=\operatorname{tr}(Z_i^\top B_i^\top B_iZ_i)
=\operatorname{tr}(B_i^\top B_iZ_iZ_i^\top)
=n\operatorname{tr}(B_i^\top B_i)
=n\operatorname{tr}(B_iB_i^\top)
=nk,
\]
and so the objective is proportional to
\begin{align*}
\|B_1Z_1-B_2Z_2\|_F^2
&=\|B_1Z_1\|_F^2-2\operatorname{tr}(Z_1^\top B_1^\top B_2Z_2)+\|B_2Z_2\|_F^2\\
&=2nk-2\operatorname{tr}((Z_2Z_1^\top)(B_1^\top B_2))
\geq 2nk-2\sum_{l\in[k]}\sigma_l(Z_2Z_1^\top),
\end{align*}
where the last step applies the von Neumann trace inequality (see Section~7.4.1 in~\cite{HornJ:12}).
This inequality is saturated when the columns of $B_1^\top$ and $B_2^\top$ are leading left and right singular vectors of $Z_1Z_2^\top$.
\end{proof}

As a consequence of Lemma~\ref{lem.proj proc prob}, we now have a fast method to solve program~\eqref{eq.mca program 2}, which we summarize in Algorithm~\ref{alg.mca}; we refer to this algorithm as matching component analysis (MCA).
(To be clear, given a matrix $A\in\mathbb{R}^{m\times n}$ of rank $r$, the \textbf{thin singular value decomposition} $A=U\Sigma V^\top$ consists of $U\in\mathbb{R}^{m\times r}$ and $V\in\mathbb{R}^{n\times r}$, both with orthonormal columns, and a diagonal matrix $\Sigma\in\mathbb{R}^{r\times r}$ of the positive singular values of $A$.)
Recalling our application, we note that matching data is an expensive enterprise, and so we wish to solve program~\eqref{eq.mca program 2} using as few samples as possible.
For this reason, we are interested in determining how many samples it takes for~\eqref{eq.mca program 2} to well approximate the original program~\eqref{eq.mca program 1}.
We summarize our study of MCA sample complexity in the remainder of this section.

\begin{algorithm}[t]
\SetAlgoLined
\KwData{$\{x_{ij}\}_{j\in[n]}$ in $\mathbb{R}^{d_i}$ for $i\in\{1,2\}$ and $k\in\mathbb{N}$}
\KwResult{$A_i\in\mathbb{R}^{k\times d_i}$ and $b_i\in\mathbb{R}^k$ for $i\in\{1,2\}$ such that 
\begin{itemize}
\item[(i)]
$\{A_ix_{ij}+b_i\}_{j\in[n]}$ has zero mean and identity covariance for both $i\in\{1,2\}$, and
\item[(ii)]
$A_1x_{1j}+b_1\approx A_2x_{2j}+b_2$ for every $j\in[n]$.
\end{itemize}}
\textit{Step 1: Normalize the data.}\\
\For{$i\in\{1,2\}$}{
Put $\overline{x}_i=\frac{1}{n}\sum_{j\in[n]}x_{ij}$, $\Sigma_i=\frac{1}{n}\sum_{j\in[n]}(x_{ij}-\overline{x}_i)(x_{ij}-\overline{x}_i)^\top$, and $r_i=\operatorname{rank}\Sigma_i$.\\
Compute thin singular value decomposition $V_i\Lambda_i V_i^\top$ of $\Sigma_i$.\\
Construct $r_i\times n$ matrix $Z_i$ with $j$th column given by $\Lambda_i^{-1/2}V_i^\top(x_{ij}-\overline{x}_i)$.}
\medskip
\textit{Step 2: Solve the projection Procrustes problem.}\\
\eIf{$k>\min\{r_1,r_2\}$}{Return \texttt{INFEASIBLE}}{
Compute $k$-truncated singular value decomposition $W_1\Sigma W_2^\top$ of $Z_1 Z_2^\top$.\\
For each $i\in\{1,2\}$, put $A_i=W_i^\top \Lambda_i^{-1/2} V_i^\top$ and $b_i=-A_i\overline{x}_i$.
}
 \caption{Matching component analysis
 \label{alg.mca}}
\end{algorithm}

\subsection{Sample complexity of MCA approximation}

Given a random $X:=[X_1;X_2]\in\mathbb{R}^{d_1}\times\mathbb{R}^{d_2}$, consider the covariances
\[
\Sigma_{X_i}:=\mathbb{E}(X_i-\mathbb{E}X_i)(X_i-\mathbb{E}X_i)^\top
\]
for $i\in\{1,2\}$.
We are interested in minimizing
\[
f_X(A)
=f_X(A_1,A_2)
:=\mathbb{E}\|A_1(X_1-\mathbb{E}X_1)-A_2(X_2-\mathbb{E}X_2)\|_2^2
\]
over the subset of $V:=\mathbb{R}^{k\times d_1}\times \mathbb{R}^{k\times d_2}$ defined by
\[
S_X:=\{(A_1,A_2)\in V:A_i\Sigma_{X_i}A_i^\top=I,~i\in\{1,2\}\}.
\]
Given $n$ independent instances of $X$, we may approximate the distribution of $X$ with the uniform distribution over these $n$ independent instances, producing the random vector $\hat{X}$.
Notice that $\hat{X}_i$ has mean $\hat\mu_i$ and covariance $\hat\Sigma_i$, as defined in \eqref{eq.mean cov est}.
We therefore have the following convenient expressions for \eqref{eq.mca program 1} and \eqref{eq.mca program 2}:
\[
\eqref{eq.mca program 1}
=\min_{A\in S_X}f_X(A),
\qquad
\eqref{eq.mca program 2}
=\min_{A\in S_{\hat{X}}}f_{\hat{X}}(A).
\]
The following is our first result on MCA sample complexity:

\begin{theorem}
\label{thm.main result 1}
Fix $p\in(0,1]$.
There exists $C=C(p)>0$ such that the following holds:
Suppose $\|X-\mathbb{E}X\|_{2,\infty}\leq\beta$ almost surely and $\min_{i\in\{1,2\}}\lambda_{\operatorname{min}}(\Sigma_{X_i})\geq\sigma^2>0$.
Then for every $\epsilon\in(0,1]$, it holds that
\[
\Big|\min_{A\in S_{\hat{X}}}f_{\hat{X}}(A)-\min_{A\in S_X}f_X(A)\Big|
\leq \epsilon\cdot\frac{\beta^2}{\sigma^2}
\]
in an event of probability $\geq 1-p$, provided
\[
n
\geq C\Big((d_1+d_2)\cdot\tfrac{k}{\epsilon^2}\log(\tfrac{k}{\epsilon^2})+(\tfrac{\beta}{\epsilon\sigma})^4\cdot\log(d_1+d_2)\Big).
\]
\end{theorem}

Note that the boundedness assumption $\|X-\mathbb{E}X\|_{2,\infty}\leq\beta$ is reasonable in practice since, for example, black-and-white images have pixel values that range from 0 to 255.
Also, we may assume $\lambda_{\operatorname{min}}(\Sigma_{X_i})>0$ without loss of generality by restricting $\mathbb{R}^{d_i}$ to the image of $\Sigma_{X_i}$ if necessary.
We prove this theorem in Section~\ref{sec.proof of first main result} using ideas from matrix analysis and high dimensional probability.

\subsection{Conditions for exact matching}
\label{subsec.exact match}

Next, we consider a family of random vectors that are particularly well suited for matching component analysis.
Suppose our probability space $(\Omega,\mathcal{F},\mathbb{P})$ takes the form $(\mathbb{R}^D,\mathcal{B},\mathbb{P})$ for some unknown $D\in\mathbb{N}$.
We say $X\in\mathbb{R}^d$ is an \textbf{affine linear random vector}
if there exists $S\in\mathbb{R}^{d\times D}$ and $\mu\in\mathbb{R}^{d}$ such that $X(\omega)=S\omega+\mu$ for every $\omega\in\mathbb{R}^D$.
While every random vector can be viewed as an affine linear random vector over the appropriate probability space, we will be interested in relating two affine linear random vectors over a common probability space.
Since $D$ and $\mathbb{P}$ are both unknown, we may assume without loss of generality that $\omega$ has zero mean and identity covariance in $\mathbb{R}^D$, and so $X$ has mean $\mu$ and covariance $SS^\top$.
%In what follows, it will be convenient to further assume that $\mathbb{P}$ is continuous.

Let $X_1$ and $X_2$ be affine linear random vectors, and suppose we encounter affine linear functions $g_1$ and $g_2$ such that $g_1(X_1)=g_2(X_2)$.
Then $g_i(X_i(\omega))$ determines $\omega$ up to a coset of some subspace $K\subseteq\mathbb{R}^D$, and the smaller this subspace is, the better we can predict $Y(\omega)$.
As one might expect, there is a limit to how small $K$ can be:

\begin{lemma}
\label{lem.kernel limit}
Suppose $X_i(\omega)=S_i\omega+\mu_i$ for each $i\in\{1,2\}$.
Then $A_1X_1+b_1=A_2X_2+b_2$ implies $A_1S_1=A_2S_2=:T$, which in turn implies $\operatorname{ker}T\supseteq\operatorname{ker}S_1+\operatorname{ker}S_2$.
\end{lemma}

\begin{proof}
Suppose $A_1X_1+b_1=A_2X_2+b_2$.
Since
\[
(A_iX_i(\omega)+b_i)-(A_iX_i(0)+b_i)
=A_iS_i\omega,
\]
it follows that $A_1S_1=A_2S_2$.
For each $i\in\{1,2\}$, we have $T=A_iS_i$, and so $\operatorname{ker}S_i\subseteq\operatorname{ker}T$.
Since $\operatorname{ker}T$ is closed under addition, the result follows.
\end{proof}

\begin{definition}
\label{eq.alm}
Given $d_1,d_2,n\in\mathbb{N}$, the corresponding \textbf{affine linear model} $\operatorname{ALM}(d_1,d_2,n)$ receives a distribution $\mathbb{P}$ over some real vector space $\mathbb{R}^D$ and returns the random function
\[
\mathcal{E}_{\mathbb{P}}\colon(S_1,\mu_1,S_2,\mu_2)\mapsto\{S_i\omega_j+\mu_i\}_{i\in\{1,2\},j\in[n]}
\]
defined over all $S_i\in\mathbb{R}^{d_i\times D}$ and $\mu_i\in\mathbb{R}^{d_i}$, and with $\{\omega_j\}_{j\in[n]}$ drawn independently with distribution $\mathbb{P}$.
We say $\operatorname{ALM}(d_1,d_2,n)$ is \textbf{exactly matchable} if there exists a measurable function
\[
\mathcal{D}\colon\{x_{ij}\}_{i\in\{1,2\},j\in[n]}\mapsto(A_1,b_1,A_2,b_2)
\]
such that for every $D\in\mathbb{N}$, every continuous probability distribution $\mathbb{P}$ over $\mathbb{R}^D$, and every input $(S_1,\mu_1,S_2,\mu_2)$, the random tuple
\[
(A_1,b_1,A_2,b_2):=(\mathcal{D}\circ\mathcal{E}_{\mathbb{P}})(S_1,\mu_1,S_2,\mu_2)
\]
almost surely satisfies both
\begin{itemize}
\item[(i)]
$A_1(S_1\omega+\mu_1)+b_1=A_2(S_2\omega+\mu_2)+b_2$ for all $\omega\in\mathbb{R}^D$, and
\item[(ii)]
$\operatorname{ker}A_iS_i=\operatorname{ker}S_1+\operatorname{ker}S_2$.
\end{itemize}
\end{definition}

%\begin{definition}
%\label{eq.alm}
%Given a distribution $\mathbb{P}$ over $\mathbb{R}^D$ and $d_1,d_2,n\in\mathbb{N}$, the corresponding \textbf{affine linear model} $\operatorname{ALM}(\mathbb{P},d_1,d_2,n)$ is the random function
%\[
%\mathcal{E}\colon(S_1,\mu_1,S_2,\mu_2)\mapsto\{S_i\omega_j+\mu_i\}_{i\in\{1,2\},j\in[n]}
%\]
%defined over all $S_i\in\mathbb{R}^{d_i\times D}$ and $\mu_i\in\mathbb{R}^{d_i}$, and with $\{\omega_j\}_{j\in[n]}$ drawn independently with distribution $\mathbb{P}$.
%We say $\operatorname{ALM}(\mathbb{P},d_1,d_2,n)$ is \textbf{exactly matchable} if there exists a deterministic function
%\[
%\mathcal{D}\colon\{x_{ij}\}_{i\in\{1,2\},j\in[n]}\mapsto(A_1,b_1,A_2,b_2)
%\]
%with the following property:
%For every $(S_1,\mu_1,S_2,\mu_2)$, the random tuple
%\[
%(A_1,b_1,A_2,b_2):=(\mathcal{D}\circ\mathcal{E})(S_1,\mu_1,S_2,\mu_2)
%\]
%almost surely satisfies both
%\begin{itemize}
%\item[(i)]
%$A_1(S_1\omega+\mu_1)+b_1=A_2(S_2\omega+\mu_2)+b_2$ for all $\omega\in\mathbb{R}^D$, and
%\item[(ii)]
%$\operatorname{ker}A_iS_i=\operatorname{ker}S_1+\operatorname{ker}S_2$.
%\end{itemize}
%\end{definition}

Our second result on MCA sample complexity provides a sharp phase transition for the affine linear model to be exactly matchable: 

\begin{theorem}\
\label{thm.main result}
\begin{itemize}
\item[(a)]
If $n\geq d_1+d_2+1$, then $\operatorname{ALM}(d_1,d_2,n)$ is exactly matchable.
\item[(b)]
If $n<d_1+d_2+1$, then $\operatorname{ALM}(d_1,d_2,n)$ is not exactly matchable.
\end{itemize}
% there does not exist a function
%\[
%\mathcal{D}_n\colon\{x_{ij}\}_{i\in\{1,2\},j\in[n]}\mapsto(A_1,b_1,A_2,b_2)
%\]
%with the following property:
%For every $(S_1,\mu_1,S_2,\mu_2)$, it holds that $(A_1,b_1,A_2,b_2):=(\mathcal{D}_n\circ\mathcal{E}_n)(S_1,\mu_1,S_2,\mu_2)$ almost surely satisfies both $A_1(S_1\omega+\mu_1)+b_1=A_2(S_2\omega+\mu_2)+b_2$ and $\operatorname{ker}A_iS_i=\operatorname{ker}S_1+\operatorname{ker}S_2$.
%\item[(b)]
%If $n\geq d_1+d_2+1$,
%\textcolor{blue}{use Algorithm~\ref{alg.mca} to define $\mathcal{D}_n$ and win.}
%\end{itemize}
\end{theorem}

In particular, we use MCA to define a witness $\mathcal{D}$ for Theorem~\ref{thm.main result}(a).
We prove this theorem in Section~\ref{sec.proof of main result 2} using ideas from matrix analysis and algebraic geometry.

\section{Experiments}
\label{sec.experiments}

In this section, we perform several experiments to evaluate the efficacy of matching component analysis for transfer learning (see Table~\ref{table} for a summary).
For each experiment, in order to produce a matching set, we take an \textbf{example set} of labeled points from the testing domain and match them with members of the conventional training set.
(While the example set resides in the testing domain, it is disjoint from the test set in all of our experiments.)
Each experiment is described by the following features; see Figure~\ref{figure.venn} for an illustration.
\begin{itemize}
\item[] \textbf{training domain.}
Space where the conventional training set resides.
\item[] \textbf{testing domain.}
Space where the example and test sets reside.
\item[] \textbf{match.}
Method used to identify a matching set, which is comprised of pairs of points from the conventional training and example sets.
%Specifically, we match each of \textbf{\textit{n}} labeled points in the testing domain to \textbf{\textit{r}} different points in the training domain.
\item[] \textbf{\textit{n}.}
Size of example set.
\item[] \textbf{\textit{r}.}
Number of points from the conventional training set that are matched to each member of the example set, producing a matching set of size $nr$.
(While our theory assumes $r=1$, we find that taking $r>1$ is sometimes helpful in practice.)
\item[] \textbf{\textit{k}.}
Parameter selected for matching component analysis.
\end{itemize}

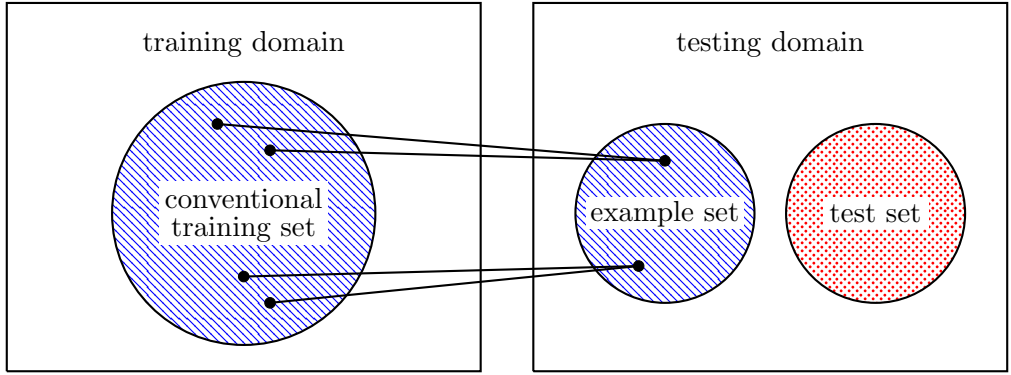
\begin{figure}
\begin{center}
\begin{tabular}{ll}
\begin{tikzpicture}[scale=0.7]
\coordinate (0) at (0-5,0);
\coordinate (3) at (0-5,3.23);
\coordinate (7) at (0-5,0.3);
\coordinate (8) at (0-5,-0.3);
\draw [thick] (-4.5-5,-3) -- (-4.5-5,4) -- (4.5-5,4) -- (4.5-5,-3) -- (-4.5-5,-3);
\draw [thick, pattern=north west lines, pattern color=blue] (0) circle [radius = 2.5];
\path [fill=white] (-5-1.6,0-0.6) rectangle (-5+1.6,0+0.6);
\node at (3){\small{training domain}};
\node at (7){\small{conventional}};
\node at (8){\small{training set}};
\coordinate (11) at (-2+5,0);
\coordinate (12) at (2+5,0);
\coordinate (13) at (0+5,3.23);
\coordinate (14) at (2+5,0);
\draw [thick] (-4.5+5,-3) -- (-4.5+5,4) -- (4.5+5,4) -- (4.5+5,-3) -- (-4.5+5,-3);
\draw [thick, pattern=north west lines, pattern color=blue] (11) circle [radius = 1.7];
\draw [thick, pattern=crosshatch dots, pattern color=red] (12) circle [radius = 1.7];
\path [fill=white] (3-1.5,0-0.3) rectangle (3+1.5,0+0.3);
\path [fill=white] (7-1,0-0.3) rectangle (7+1,0+0.3);
\node at (11){\small{example set}};
\node at (12){\small{test set}};
\node at (13){\small{testing domain}};
\coordinate (001) at (-4.5,1.2);
\coordinate (002) at (-5.5,1.7);
\coordinate (003) at (-5,-1.2);
\coordinate (004) at (-4.5,-1.7);
\coordinate (101) at (3,1);
\coordinate (102) at (2.5,-1);
\draw [fill] (001) circle [radius=0.1];
\draw [fill] (002) circle [radius=0.1];
\draw [fill] (003) circle [radius=0.1];
\draw [fill] (004) circle [radius=0.1];
\draw [fill] (101) circle [radius=0.1];
\draw [fill] (102) circle [radius=0.1];
\draw [thick] (001) -- (101);
\draw [thick] (002) -- (101);
\draw [thick] (003) -- (102);
\draw [thick] (004) -- (102);
\end{tikzpicture}
\end{tabular}
\end{center}
\caption{
\label{figure.venn}
\footnotesize{
Illustration of experimental setup in Section~\ref{sec.experiments}.
The goal is to train a classifier on a training set that performs well on a test set.
The training set, depicted in blue hatching, consists of both a conventional training set in the training domain and a small example set in the testing domain.
The test set, depicted in red dots, is unknown and resides in the testing domain.
Importantly, the example set is disjoint from the test set despite both residing in the testing domain.
We match each member of the example set to $r$ members of the conventional set to produce a matching set.
(In the above illustration, $r=2$.)
This matching set is then processed by MCA to identify mappings that send both the training domain and the testing domain to a common domain.
}}
\label{fig:roux}
\end{figure}

For each experiment, we run MCA to find affine linear mappings to the common domain $\mathbb{R}^k$, and then we train a k-nearest neighbor (k-NN) classifier in this domain on the conventional training set, and we test by first mapping the test set into the common domain.
For comparison, we consider two different baselines, which we denote by BL1 and BL2.
For BL1, we train a k-NN classifier on the example set (whose size is only $n$) and test on the test set.
For BL2, we train a k-NN classifier on the conventional training set (which resides in the training domain $\mathbb{R}^{d_1}$) and test on the test set (which resides in the testing domain $\mathbb{R}^{d_2}$).
This latter baseline is possible whenever $d_1=d_2$, which occurs in all of our experiments.
In order to isolate the performance of MCA in our experiments, we set the number of neighbors to be 10 for all of our k-NN classifiers.

In half of the experiments we consider, we are given a matching set with $r=1$, and in the other experiments, we are only given an example set.
In this latter case, we have the luxury of selecting $r$, and in both cases, we have the additional luxury of selecting $k$.
We currently do not have a rule of thumb for selecting these parameters, although we observe that overall performance is sensitive to the choice of parameters.
See Section~4 for more discussion along these lines.

\begin{table}[t]
\centering
\caption{Summary of transfer learning experiments\label{table}}
\medskip
\begin{tabular}{ | c c c c c c c c c | }\hline
\textbf{training domain} & \textbf{testing domain} & \textbf{match} & \textbf{\textit{n}} & \textbf{\textit{r}} & \textbf{\textit{k}} & \textbf{BL1} & \textbf{BL2} & \textbf{MCA} \\\hline
MNIST (1st half) & MNIST (2nd half) & $\ell^2$ & 2000 & 5 & 30 & 86\% & \textbf{94\%} & 90\%\\
MNIST (1st half) & MNIST (2nd half) & label & 2000 & 5 & 30 & 86\% & \textbf{94\%} & 69\%\\\hline
MNIST (crop) & MNIST (pixelate) & source & 20 & 1 & 19 & 18\% & 23\% & \textbf{83\%}\\
MNIST (crop) & MNIST (pixelate) & source & 2000 & 1 & 50 & 91\% & 23\% & \textbf{94\%}\\\hline
CF (2 \& 5) & MNIST (2 \& 5) & $\ell^2$ & 10 & 100 & 5 & 54\% & \textbf{98\%} & 84\%\\
CF (0 \& 1) & MNIST (0 \& 1) & $\ell^2$ & 10 & 100 & 5 & 55\% & \textbf{100\%} & \textbf{100\%}\\
CF (4 \& 9) & MNIST (4 \& 9) & $\ell^2$ & 10 & 100 & 5 & 51\% & \textbf{89\%} & 71\%\\\hline
SAMPLE (sim) & SAMPLE (real) & expert & 100 & 1 & 99 & 62\% & 20\% & \textbf{87\%} \\\hline
\end{tabular}
\end{table}

\subsection{Transfer learning from MNIST to MNIST}

For our first experiment, we tested the performance of the MCA algorithm in a seemingly trivial case: when the training and testing domains are identical.
Of course, the MCA algorithm should not outperform the baseline BL2 in this simple case.
However, this setup allows us to isolate the impact of using different matching procedures.  

We partitioned the training set of 60,000 MNIST digits into two subsets of equal size.
We arbitrarily chose the first 30,000 to represent the training domain, and interpreted the remaining 30,000 points as members of the testing domain.
We then matched $n$ of the points from the testing domain with $r=5$ of their nearest neighbors (in the Euclidean sense) in the training domain with the same label.
For a cheaper alternative, we also tried matching with $r=5$ randomly selected members of the training domain that have the same label.

As expected, MCA does not outperform the classifier trained on the entire training set (BL2).
However, with sufficiently many matches, MCA is able to find a low-dimensional embedding of $\mathbb{R}^{28\times 28}$ that still allows for accurate classification of digits.
Judging by the poor performance of the label-based matching, these experiments further illustrate the importance of a thoughtful matching procedure.
In general, when label classes exhibit large variance and yet the matching is determined by label information alone, we observe that MCA often fails to identify a common domain that allows for transfer learning.

\subsection{Transfer learning from cropped MNIST to pixelated MNIST}

Our second experiment replicates the affine linear setup from Subsection~\ref{subsec.exact match}.
Here, we view the MNIST dataset as a subset $M$ of a probability space $\Omega=\mathbb{R}^{28\times 28}$ with $\mathbb{P}$ distributed uniformly over $M$.
Next, we linearly transform the MNIST dataset by applying two different maps $\omega\mapsto X_i(\omega)$.
In particular, $X_1(\cdot)$ crops a given $28\times 28$ image to the middle $14\times 14$ portion, while $X_2(\cdot)$ forms a $14\times14$ pixelated version of the original image by averaging over each $2\times 2$ block; see Figure~\ref{fig.affine-linear} for an illustration.
We interpret the cropped images $\{X_1(\omega)\}_{\omega\in M}$ as belonging to the training domain and the pixelated images $\{X_2(\omega)\}_{\omega\in M}$ to the testing domain.
Notice that this setup delivers a natural matching between members of both domains, i.e., $X_1(\omega)$ is matched with $X_2(\omega)$ for every $\omega\in M$; as such, $r=1$.
We evaluate the performance of MCA against the baselines with both $n=20$ and $n=2000$.
These experiments are noteworthy because MCA beats both baselines for both small and large values of $n$.
We credit this behavior to the affine linear setup, since in general, we find that MCA beats BL1 only when $n$ is small.
See Figure~\ref{fig.affine-linear} for a visualization of the information captured in the common domain.

\begin{figure}[t]
\centering
\includegraphics[width=0.7\textwidth]{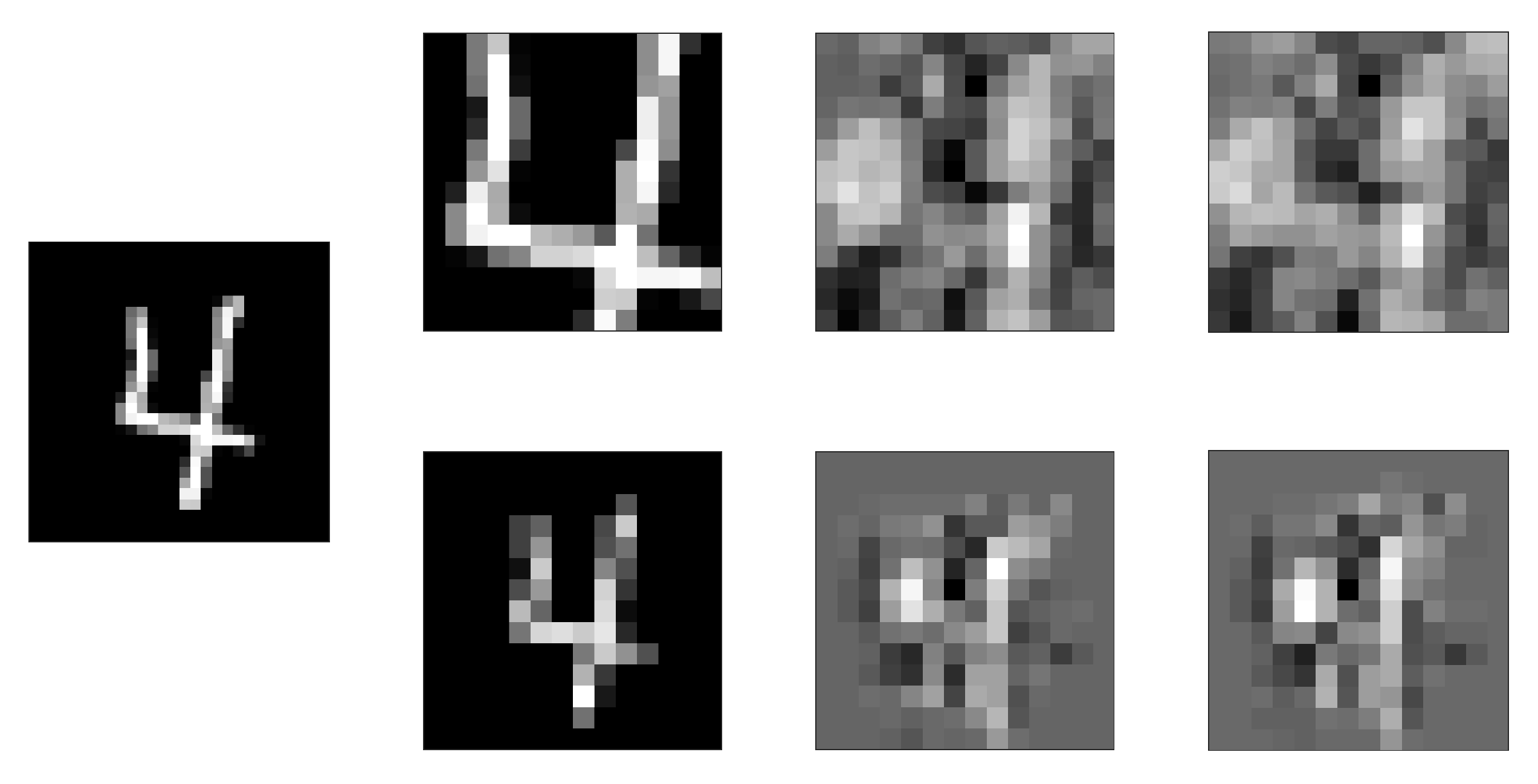}
\caption{\footnotesize{Transfer learning from cropped MNIST digits to pixelated MNIST digits.
We crop each $28\times28$ MNIST digit to its middle $14\times 14$ portion.
We also form a $14\times 14$ pixelated version of each MNIST digit by averaging over $2\times 2$ blocks. 
For example, \textbf{(left)} depicts a 4 from the MNIST test set, while \textbf{(middle left)} depicts both cropped and pixelated versions of the same 4.
%Notice that both the crop and pixelate mappings are affine linear (in fact, linear), and so considering the discussion in Subsection~\ref{subsec.exact match}, we expect MCA to perform well.
%We take a matching set of size $n=20$, and as a first baseline, we train a k-NN classifier on these 20 pixelated MNIST digits, resulting in an accuracy of 18\% on the pixelated test set.
%For a second baseline, we train a k-NN classifier on 48,000 cropped MNIST digits.
%In this case, the accuracy is 23\% on the pixelated test set.
We run MCA with $k=19$ to identify a common domain.
We provide two illustrations of the information captured in the common domain.
\textbf{(middle right)}
For an image in domain $i\in\{1,2\}$, we apply the MCA-learned affine-linear map $g_i$ to send the image to the common domain, and then apply the pseudoinverse of $g_i$ to return the image back to domain $i$.
\textbf{(right)}
For an image in domain $i\in\{1,2\}$, we apply the MCA-learned affine-linear map $g_i$ to send the image to the common domain, and then apply the pseudoinverse of $g_{i'}$ to send the image to the other domain $i':=3-i$.
The fact that these projections look so similar illustrates that MCA identified well-matched components.
}
\label{fig.affine-linear}}
\end{figure}

\subsection{Transfer learning from computer fonts to MNIST}

For this experiment, we attempted transfer learning from the computer font (CF) digits provided in~\cite{Bernhardsson:online} to MNIST digits.
While the MNIST digits are $28\times 28$, the CF digits are $64\times 64$.
In order to put both into a common domain, we resized both datasets to be $16\times 16$; see Figure~\ref{fig.mnist-computer} for an illustration.
Interestingly, resizing MNIST in this way makes BL1 succeed with even modest values of $n$.
In order to make MCA competitive, we decided to focus on binary classification tasks, specifically, classifying 2 vs.\ 5, 0 vs.\ 1, and 4 vs.\ 9.
To identify a matching between CF and MNIST digits, we looked for $r=100$ CF digits that were closest to each of the $n$ MNIST digits in the Euclidean distance.
(For runtime considerations, we first selected 5,000 out of the 56,443 computer fonts that tended to be close to MNIST digits, and then limited our search to digits in these fonts.)
Since we used the Euclidean distance for matching, it comes as no surprise that BL2 outperforms MCA.
While Table~\ref{table} details the $n=10$ case, Figure~\ref{fig.mnist-computer} illustrates performance for each $n\in[10:10:150]$.
Surprisingly, the performance of MCA drops for larger values of $n$.
We discuss this further in Section~4.

\begin{figure}[t]
\centering
\includegraphics[width=0.98\textwidth]{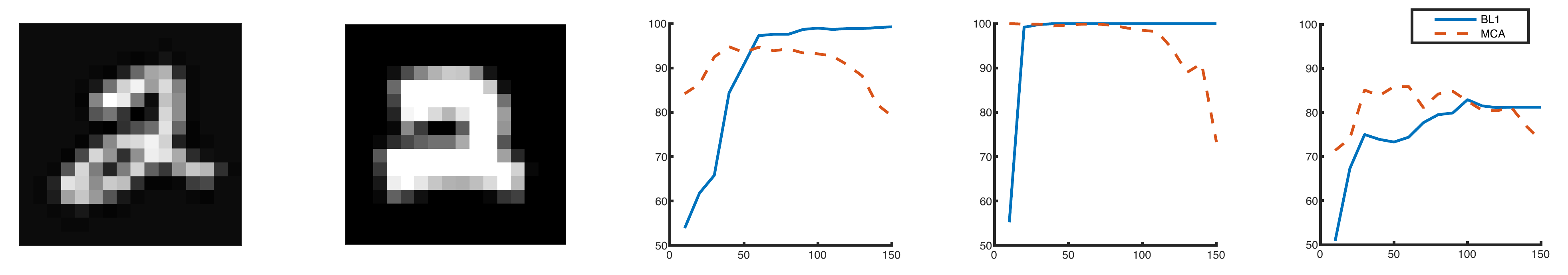}
\hspace{-3.88in}\rotatebox{90}{\tiny{~~\quad\quad test accuracy}}\hspace{3.73in}\tiny{$n$}
\caption{\footnotesize{Transfer learning from computer font digits~\cite{Bernhardsson:online} to MNIST digits.
We train binary classifiers for 2 vs.\ 5, 0 vs.\ 1, and 4 vs.\ 9.
In each setting, select $n\in[10:10:150]$, and draw $n$ MNIST digits at random.
For each of these digits, find the $r=100$ closest computer font digits in the Euclidean distance.
An example of a match is depicted in \textbf{(left)} and \textbf{(middle left)}.
As a baseline, we train a k-NN classifier on the MNIST portion of the matching set.
We also run MCA on the matching set with $k=5$, and then train a k-NN classifier on the common domain.
The accuracy of these classifiers on the test set is depicted in \textbf{(middle)} for 2 vs.\ 5, in \textbf{(middle right)} for 0 vs.\ 1, and in \textbf{(right)} for 4 vs.\ 9.
}
\label{fig.mnist-computer}}
\end{figure}

\subsection{Transfer learning with the SAMPLE dataset}

Finally, we consider transfer learning with the Synthetic and Measured Paired and Labeled (SAMPLE) database of computer-simulated and real-world SAR images~\cite{LewisEtal:19}. 
The publicly-available SAMPLE database consists of 1366 paired images of 10 different vehicles, each pair consisting of a real-world SAR image and a corresponding computer-simulated SAR image; see Figure~\ref{fig.sar} for an illustration. 

In this experiment, the training domain corresponds to simulated data, and the testing domain corresponds to real-world data.
The training set consists of 80\% of the simulated set of SAMPLE images, $n=100$ of which are matched with corresponding real-world data.
The test set consists of the real-world data corresponding to the withheld 20\% of simulated training set.
In this case, MCA substantially out-performs both BL1 and BL2; see Figure~\ref{fig.sar} for a depiction of the normalized confusion matrices in these cases.
We note that BL2 is similar to the SAR classification challenge problem outlined in~\cite{LewisEtal:19} and~\cite{scarnati2019deep}, where a small convolutional neural network (CNN) achieved 24\% accuracy, and a densely connected CNN achieved 55\% accuracy.
Impressively, by mapping to the common domain identified by MCA, we can simply use a k-NN classifier and increase performance to 87\%. 

\begin{figure}[t]
\centering
\includegraphics[width=\textwidth]{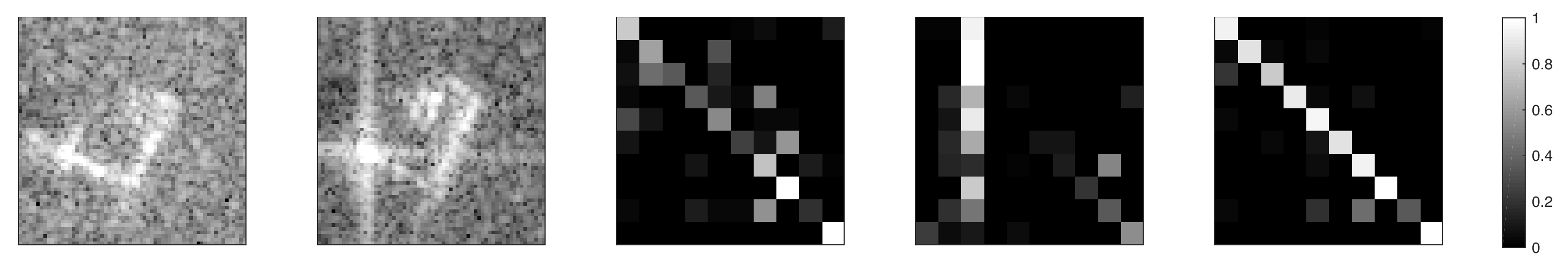}%\\
%(a)
%\hspace{0.92in}
%(b)
%\hspace{0.92in}
%(c)
%\hspace{0.92in}
%(d)
%\hspace{0.92in}
%(e)
%\hspace{0.4in}
\caption{\footnotesize{Transfer learning with the SAMPLE database of computer-simulated and real-world SAR images~\cite{LewisEtal:19}.
The SAMPLE database consists of 1366 paired images of 10 different vehicles, each pair consisting of a real-world SAR image and a corresponding computer-simulated SAR image.
For example, \textbf{(left)} is a real-world SAR image of an M548 tracked cargo carrier, while \textbf{(middle left)} is a corresponding computer-simulated SAR image that was developed with the help of a CAD model of the M548.
Our goal is to use 80\% (1092) of the computer-simulated images, 100 of which are paired with corresponding real-world images, to train a classifier that performs well on a test set comprised of the other 20\% (274) of real-world images.
\textbf{(middle)}
For a first baseline, we train a k-NN classifier on the 100 real-world images.
We depict the resulting normalized confusion matrix over the test set.
For this baseline, 62\% of the test set is classified correctly.
\textbf{(middle right)}
For a second baseline, we train a k-NN classifier on the 1092 computer-simulated images.
In this case, the classifier labels most images as the third vehicle type, namely, the BTR-70 armored personnel carrier.
Only 20\% of the test set is classified correctly.
\textbf{(right)}
Finally, we run matching component analysis (Algorithm~\ref{alg.mca}) with $k=99$ on the 100 paired images to identify a common domain, and then we train a k-NN classifier on the 1092 computer-simulated images in this common domain.
For this alternative, 87\% of the test set is classified correctly in the common domain.}
\label{fig.sar}}
\end{figure}

\section{Discussion}

This paper introduced matching component analysis (MCA, Algorithm~\ref{alg.mca}) as a method for identifying features in data that are appropriate for transfer learning.
In this section, we reflect on our observations and identify various opportunities for future work.

The theory developed in this paper concerned the sample complexity of MCA.
The fundamental question to answer is
\begin{center}
\textit{How large of a matching set is required to perform high-accuracy transfer learning?}
\end{center}
In order to isolate the performance of MCA, our theory does not rely on the choice of the classifier, and because of this, our sample complexity results rely on different proxies for success.
Overall, a different approach is needed to answer the above question.

Like many algorithms in machine learning, MCA requires the user to select a parameter, namely, $k$.
We currently do not have a rule of thumb for selecting this parameter.
Also, one should expect that a larger matching set will only help with transfer learning, but some of our experiments seem to suggest that MCA behaves \textit{worse} given more matches (see Figure~\ref{fig.mnist-computer}, for example).
While we do not understand this behavior, one can get around this by partitioning the matching set into batches, training a weak classifier on each batch, and then boosting.
The drop in performance might reflect the fact that MCA is oblivious to the data labels, suggesting a label-aware alternative (cf.\ PCA vs.\ SqueezeFit~\cite{McWhirterMV:18}).
The performance drop might also reflect our choice of affine linear maps and Euclidean distances, suggesting alternatives involving non-linear maps and other distances.

As one would expect, transfer learning is more difficult when the matching set is poorly matched.
Indeed, we observed this when transfer learning from MNIST to MNIST using two different matching techniques.
In practice, it is expensive to find a good matching set.
For example, for the SAMPLE dataset~\cite{LewisEtal:19}, it took two years of technical expertise to generate accurate computer-simulated matches.
In general, one might attempt to automate the matching process with an algorithm such as GHMatch~\cite{VillarBBW:16}, but we find that runtimes are slow for even moderately large datasets; e.g., it takes several minutes to match datasets with more than 50 points.
Overall, finding a matching set appears to be a bottleneck, akin to finding labels for a training set.
As an alternative, it would be interesting to instead develop theory that allows for transfer learning given non-matched data in both domains without having to first match the data.

\section{Proof of Theorem~\ref{thm.main result 1}}
\label{sec.proof of first main result}

It is convenient to define the diagonal operator
\[
D:=\left[\begin{array}{cc}I_{d_1}&0\\0&-I_{d_2}\end{array}\right]
\]
so that our objective function takes the form
\[
f_X(A)
=f_X(A_1,A_2)
:=\mathbb{E}\|A_1(X_1-\mathbb{E}X_1)-A_2(X_2-\mathbb{E}X_2)\|_2^2
=\mathbb{E}\|AD(X-\mathbb{E}X)\|_2^2.
\]
In what follows, we let $\|\cdot\|_V$ denote the norm on $V$ defined by
\[
\|(A_1,A_2)\|_V:=\max\{\|A_1\|_{2\to2},\|A_2\|_{2\to2}\}.
\]
This determines a Hausdorff distance $\operatorname{dist}$ between nonempty subsets of $V$.
Throughout, we denote $T_\alpha:=\{A\in V:\|A\|_V\leq \alpha\}$.
Our approach is summarized in the following lemma:

\begin{lemma}
\label{lem.key lemma}
Let $X,Y\in \mathbb{R}^{d_1}\times\mathbb{R}^{d_2}$ be random vectors such that
\begin{itemize}
\item[(i)]
$\operatorname{dist}(S_X,S_Y)\leq\epsilon_1$,
\item[(ii)]
$f_X,f_Y\colon(S_X\cup S_Y,\|\cdot\|_V)\to\mathbb{R}$ are both $L$-Lipschitz, and
\item[(iii)]
$|f_X(A)-f_Y(A)|\leq \epsilon_2$ for every $A\in S_X\cup S_Y$.
\end{itemize}
Then $\displaystyle\Big|\min_{A\in S_X}f_X(A)-\min_{A\in S_Y}f_Y(A)\Big|\leq L\epsilon_1+\epsilon_2$.
\end{lemma}

\begin{proof}
Without loss of generality, it holds that $\min_{A\in S_X}f_X(A)\geq\min_{A\in S_Y}f_Y(A)$.
Let $A^\star$ denote an optimizer for $f_Y$.
By (i), there exists $B\in S_X$ such that $\|B-A^\star\|_V\leq \epsilon_1$, and then by (ii), it holds that $f_X(B)\leq f_X(A^\star)+L\epsilon_1$.
As such,
\[
\Big|\min_{A\in S_X}f_X(A)-\min_{A\in S_Y}f_Y(A)\Big|
\leq f_X(B)-f_Y(A^\star)
\leq L\epsilon_1+f_X(A^\star)-f_Y(A^\star)
\leq L\epsilon_1+\epsilon_2,
\]
where the last step applies (iii).
\end{proof}

As such, it suffices to show that $\hat{X}$ and $X$ satisfy Lemma~\ref{lem.key lemma}(i)--(iii).
In order to verify Lemma~\ref{lem.key lemma}(i), it is helpful to have a bound on the members of $S_X$:

\begin{lemma}
\label{lem.SX is bounded}
Suppose $\Sigma\succ0$.
If $A\Sigma A^\top=I$, then $\|A\|_{2\to2}^2\leq\lambda_{\operatorname{min}}(\Sigma)^{-1}$.
\end{lemma}

\begin{proof}
First, we observe that
\[
1
=\|I\|_{2\to2}
=\|A\Sigma A^\top\|_{2\to2}
=\|\Sigma^{1/2}A^\top\|_{2\to2}^2.
\]
Next, select a unit vector $x$ such that $\|A^\top x\|_2=\|A\|_{2\to2}$.
Then
\[
\|\Sigma^{1/2}A^\top\|_{2\to2}
\geq\|\Sigma^{1/2}A^\top x\|_2
\geq\lambda_{\operatorname{min}}(\Sigma^{1/2})\cdot\|A^\top x\|_2
=\lambda_{\operatorname{min}}(\Sigma^{1/2})\cdot\|A\|_{2\to2}.
\]
The result then follows by combining and rearranging the above estimates.
\end{proof}

\begin{lemma}
\label{lem.bound haus dist}
Suppose $\Sigma_{X_i},\Sigma_{Y_i}\succ0$ for both $i\in\{1,2\}$.
Then
\[
\operatorname{dist}(S_X,S_Y)^2
\leq\max_{i\in\{1,2\}}\frac{\|\Sigma_{X_i}-\Sigma_{Y_i}\|_{2\to2}}{\lambda_{\operatorname{min}}(\Sigma_{X_i})\cdot\lambda_{\operatorname{min}}(\Sigma_{Y_i})}.
\]
\end{lemma}

\begin{proof}
Define the function $g_{XY}\colon V\to V$ by
\[
g_{XY}(A_1,A_2)
:=(A_1\Sigma_{X_1}^{1/2}\Sigma_{Y_1}^{-1/2},A_2\Sigma_{X_2}^{1/2}\Sigma_{Y_2}^{-1/2}).
\]
Observe that $g_{XY}$ maps every point $(A_1,A_2)\in S_X$ to a point in $S_Y$:
\[
(A_i\Sigma_{X_i}^{1/2}\Sigma_{Y_i}^{-1/2})\Sigma_{Y_i}(A_i\Sigma_{X_i}^{1/2}\Sigma_{Y_i}^{-1/2})^\top
=A_i\Sigma_{X_i}A_i^\top
=I.
\]
Furthermore, for every $(A_1,A_2)\in S_X$, we may apply sub-multiplicativity, Lemma~\ref{lem.SX is bounded}, and then Theorem~X.1.1 in~\cite{Bhatia:13} to obtain
\begin{align*}
\|A_i\Sigma_{X_i}^{1/2}\Sigma_{Y_i}^{-1/2}-A_i\|_{2\to2}^2
&=\|A_i(\Sigma_{X_i}^{1/2}-\Sigma_{Y_i}^{1/2})\Sigma_{Y_i}^{-1/2}\|_{2\to2}^2\\
&\leq\|A_i\|_{2\to2}^2\cdot\|\Sigma_{X_i}^{1/2}-\Sigma_{Y_i}^{1/2}\|_{2\to2}^2\cdot\|\Sigma_{Y_i}^{-1/2}\|_{2\to2}^2\\
&\leq\frac{\|\Sigma_{X_i}^{1/2}-\Sigma_{Y_i}^{1/2}\|_{2\to2}^2}{\lambda_{\operatorname{min}}(\Sigma_{X_i})\cdot\lambda_{\operatorname{min}}(\Sigma_{Y_i})}
\leq\frac{\|\Sigma_{X_i}-\Sigma_{Y_i}\|_{2\to2}}{\lambda_{\operatorname{min}}(\Sigma_{X_i})\cdot\lambda_{\operatorname{min}}(\Sigma_{Y_i})}.
\end{align*}
Maximizing over $i\in\{1,2\}$ produces an upper bound on $\sup_{A\in S_X}\|g_{XY}(A)-A\|_V^2$.
By symmetry, the same bound holds for $\sup_{A\in S_Y}\|g_{YX}(A)-A\|_V^2$, implying the result.
\end{proof}

Overall, for Lemma~\ref{lem.key lemma}(i), it suffices to have spectral control over the covariance.
In the special case where $Y=\hat{X}$, we will accomplish this with the help of Matrix Hoeffding~\cite{MackeyEtal:12}.
Before doing so, we consider Lemma~\ref{lem.key lemma}(ii):

\begin{lemma}
\label{lem.v norm bound}
For every $A\in V$, it holds that $\|A\|_{2\to2}\leq \sqrt{2}\cdot\|A\|_V$.
\end{lemma}

\begin{proof}
Select a unit vector $x=[x_1;x_2]$ such that $\|A\|_{2\to2}=\|Ax\|_2$.
Then the triangle and Cauchy--Schwarz inequalities together give
\begin{align*}
\|A\|_{2\to2}
=\|A_1x_1+A_2x_2\|_2
&\leq \|A_1\|_{2\to2}\|x_1\|_2+\|A_2\|_{2\to2}\|x_2\|_2\\
&\leq \Big(\|A_1\|_{2\to2}^2+\|A_2\|_{2\to2}^2\Big)^{1/2}\Big(\|x_1\|_2^2+\|x_2\|_2^2\Big)^{1/2}\\
&\leq \sqrt{2}\cdot\max_{i\in\{1,2\}}\|A_i\|_{2\to2}
=\sqrt{2}\cdot\|A\|_V.
\qedhere
\end{align*}
\end{proof}

\begin{lemma}
\label{lem.lipschitz bound}
Suppose $\|X-\mathbb{E}X\|_{2,\infty}\leq \beta$ almost surely.
Then $f_X\colon(T_\alpha,\|\cdot\|_V)\to\mathbb{R}$ is $8\alpha \beta^2$-Lipschitz.
\end{lemma}

\begin{proof}
Put $Z:=X-\mathbb{E}X$ so that $f_X(A)=\mathbb{E}\|ADZ\|_2^2$, and select any $A,B\in T_\alpha$.
Then
\begin{align*}
|f_X(A)-f_X(B)|
&=\big|\mathbb{E}\|ADZ\|_2^2-\mathbb{E}\|BDZ\|_2^2\big|\\
&\leq\mathbb{E}\big|\|ADZ\|_2^2-\|BDZ\|_2^2\big|\\
&=\mathbb{E}\Big[\big(\|ADZ\|_2+\|BDZ\|_2\big)\cdot\big|\|ADZ\|_2-\|BDZ\|_2\big|\Big].
\end{align*}
To proceed, we bound each of the factors in the right-hand side.
First,
\[
\|ADZ\|_2
=\|A_1Z_1-A_2Z_2\|_2
\leq\|A_1\|_{2\to2}\|Z_1\|_2+\|A_2\|_{2\to2}\|Z_2\|_2
\leq 2\alpha\beta
\]
almost surely.
Similarly, $\|BDZ\|_2\leq 2\alpha\beta$ almost surely.
Next,
\[
\big|\|ADZ\|_2-\|BDZ\|_2\big|
\leq\|ADZ-BDZ\|_2
\leq\|A-B\|_{2\to2}\cdot\|Z\|_2
\leq 2\beta\cdot\|A-B\|_V
\]
almost surely, where the last step follows from Lemma~\ref{lem.v norm bound}.
Combining these estimates then gives the result.
\end{proof}

Our approach for demonstrating Lemma~\ref{lem.key lemma}(iii) is an net-based argument that is specialized to the case where $Y=\hat{X}$.
Our choice of net is a modification of what is used to estimate the spectral norm of subgaussian matrices:

\begin{lemma}
\label{lem.net exists}
Fix $\alpha,\eta>0$.
There exists $N\subseteq T_{\alpha+\eta}$ such that
\begin{itemize}
\item[(i)]
for every $x\in T_\alpha$, there exists $y\in N$ such that $\|x-y\|_V\leq\eta$, and
\item[(ii)]
$|N|\leq(1+\frac{2\sqrt{2k}\alpha}{\eta})^{k(d_1+d_2)}$.
\end{itemize}
\end{lemma}

\begin{proof}
We will construct $N$ by first identifying an $\eta$-net $N_\eta$ for the Frobenius ball $B$ of radius $\sqrt{2k}\alpha$, and then taking $N:=N_\eta\cap T_{\alpha+\eta}$.
Indeed, Lemma~\ref{lem.v norm bound} implies
\[
\|A\|_F
\leq\sqrt{k}\cdot\|A\|_{2\to2}
\leq\sqrt{2k}\cdot\|A\|_V,
\]
and so $T_\alpha\subseteq B$.
As such, for every $x\in T_\alpha\subseteq B$, there exists $y\in N_\eta$ such that
\[
\|x-y\|_V\leq\|x-y\|_{2\to2}\leq\|x-y\|_F\leq \eta.
\]
Furthermore, this choice of $y$ necessarily resides in $T_{\alpha+\eta}$:
\[
\|y\|_V
=\|x-x+y\|_V
\leq\|x\|_V+\|x-y\|_V
\leq \alpha+\eta.
\]
As such, $N=N_\eta\cap T_{\alpha+\eta}$ satisfies (i).
A standard volume comparison argument (see Proposition~4.2.12 in~\cite{Vershynin:18}, for example) gives that $N_\eta$ satisfies the bound in (ii), and we are done by observing that $|N|\leq|N_\eta|$.
\end{proof}

The remainder of our proof is specialized to the case where $Y=\hat{X}$, and throughout, we make use of the following extensions to Hoeffding's inequality:

\begin{proposition}[Matrix Hoeffding~\cite{MackeyEtal:12}]
Suppose $\{X_j\}_{j\in[n]}$ are independent copies of a random symmetric matrix $X\in\mathbb{R}^{d\times d}$ such that $\mathbb{E}X=0$ and $\|X\|_{2\to2}\leq b$ almost surely.
Then for every $t\geq0$, it holds that
\[
\mathbb{P}\bigg\{\bigg\|\frac{1}{n}\sum_{j\in[n]}X_j\bigg\|_{2\to2}\geq t\bigg\}\leq 2d\cdot e^{-nt^2/(2b^2)}.
\]
\end{proposition}

\begin{proposition}[Vector Hoeffding]
Suppose $\{X_j\}_{j\in[n]}$ are independent copies of a random vector $X\in\mathbb{R}^d$ such that $\|X\|_2\leq b$ almost surely.
Then for every $t\geq0$, it holds that
\[
\mathbb{P}\Bigg\{\bigg\|\frac{1}{n}\sum_{j\in[n]}X_j\bigg\|_2\geq t\Bigg\}
\leq 2(d+1)\cdot e^{-nt^2/(2b^2)}.
\]
\end{proposition}

\begin{proof}
Following Section~2.1.16 in~\cite{Tropp:15}, for each column vector $v\in\mathbb{R}^d$, we consider the symmetric matrix
\[
M(v):=\left[\begin{array}{cc}0&v^\top\\v&0\end{array}\right].
\]
Then since
\[
M(v)^2=\left[\begin{array}{cc}\|v\|_2^2&0^\top\\0&vv^\top\end{array}\right],
\]
it holds that $\|M(v)\|_{2\to2}^2=\|M(v)^2\|_{2\to2}=\|v\|_2^2$.
Linearity then gives
\[
\bigg\|\frac{1}{n}\sum_{j\in[n]}X_j\bigg\|_2
=\bigg\|M\bigg(\frac{1}{n}\sum_{j\in[n]}X_j\bigg)\bigg\|_{2\to2}
=\bigg\|\frac{1}{n}\sum_{j\in[n]}M(X_j)\bigg\|_{2\to2}.
\]
By assumption, $\|M(X)\|_{2\to2}=\|X\|_2\leq b$ almost surely, and so Matrix Hoeffding implies
\[
\mathbb{P}\Bigg\{\bigg\|\frac{1}{n}\sum_{j\in[n]}X_j\bigg\|_2\geq t\Bigg\}
=\mathbb{P}\Bigg\{\bigg\|\frac{1}{n}\sum_{j\in[n]}M(X_j)\bigg\|_{2\to2}\geq t\Bigg\}
\leq2(d+1)\cdot e^{-nt^2/(2b^2)}.
\qedhere
\]
\end{proof}

For the remainder of this section, we make the following assumptions without mention:
$X=[X_1;X_2]$ is a random vector in $\mathbb{R}^{d_1}\times\mathbb{R}^{d_2}$ with mean $\mu=[\mu_1;\mu_2]$, and $\hat{X}$ is a random vector with mean $\hat\mu=[\hat\mu_1;\hat\mu_2]$ that is distributed uniformly over independent realizations $\{X_j=[X_{1j};X_{2j}]\}_{j\in[n]}$ of $X$.
It will always be clear from context whether $X_1$ refers to the first component of $X$ or the first independent copy of $X$.
We first tackle Lemma~\ref{lem.key lemma}(i) with the help of Lemma~\ref{lem.bound haus dist}: 

\begin{lemma}
\label{lem.bound spectral distance}
Suppose $\|X-\mu\|_{2,\infty}\leq\beta$ almost surely.
Then for every $\delta\geq0$, it holds that
\[
\max_{i\in\{1,2\}}\|\Sigma_{\hat{X}_i}-\Sigma_{X_i}\|_{2\to2}\leq\delta
\quad
\text{w.p.}
\quad
\geq1-8(d_1+d_2)\cdot e^{-\frac{n}{2}\cdot f(\frac{\delta}{2\beta^2})},
\]
where $f(z):=\min(z,z^2)$.
\end{lemma}

\begin{proof}
Add zero and expand to obtain
\begin{align*}
\Sigma_{\hat{X}_i}&=\frac{1}{n}\sum_{j\in[n]}(X_{ij}-\hat\mu_i)(X_{ij}-\hat\mu_i)^\top\\
&=\frac{1}{n}\sum_{j\in[n]}\Big((X_{ij}-\mu_i)-(\hat\mu_i-\mu_i)\Big)\Big((X_{ij}-\mu_i)-(\hat\mu_i-\mu_i)\Big)^\top\\
&=\frac{1}{n}\sum_{j\in[n]}(X_{ij}-\mu_i)(X_{ij}-\mu_i)^\top-(\hat\mu_i-\mu_i)(\hat\mu_i-\mu_i)^\top.
\end{align*}
The triangle inequality then gives
\[
\|\Sigma_{\hat{X}_i}-\Sigma_{X_i}\|_{2\to2}
\leq\bigg\|\frac{1}{n}\sum_{j\in[n]}\Big((X_{ij}-\mu_i)(X_{ij}-\mu_i)^\top-\Sigma_{X_i}\Big)\bigg\|_{2\to2}+\bigg\|\frac{1}{n}\sum_{j\in[n]}(X_{ij}-\mu_i)\bigg\|_2^2.
\]
For the first term, note that $\|A-B\|_{2\to2}\leq\max\{\|A\|_{2\to2},\|B\|_{2\to2}\}$ when $A,B\succeq0$, and so
\[
\Big\|(X_{ij}-\mu_i)(X_{ij}-\mu_i)^\top-\Sigma_{X_i}\Big\|_{2\to2}
\leq \max\Big\{\|X_{ij}-\mu_i\|_2^2,\|\Sigma_{X_i}\|_{2\to2}\Big\}
\leq \beta^2
\]
almost surely.
Matrix Hoeffding then gives
\[
\bigg\|\frac{1}{n}\sum_{j\in[n]}\Big((X_{ij}-\mu_i)(X_{ij}-\mu_i)^\top-\Sigma_{X_i}\Big)\bigg\|_{2\to2}
\leq \delta_1
\quad
\text{w.p.}
\quad
\geq1-2d_i\cdot e^{-n\delta_1^2/(2\beta^4)}.
\]
Next, we bound the second term by Vector Hoeffding:
\[
\bigg\|\frac{1}{n}\sum_{j\in[n]}(X_{ij}-\mu_i)\bigg\|_2
\leq \delta_2
\quad
\text{w.p.}
\quad
\geq1-2(d_i+1)\cdot e^{-n\delta_2^2/(2\beta^2)}.
\]
The result follows by setting $\delta_1=\delta_2^2=\delta/2$ and applying the union bound.
\end{proof}

%\begin{lemma}
%Suppose $\|X-\mu\|_{2,\infty}\leq\beta$ almost surely and
%\[
%\min_{i\in\{1,2\}}\lambda_{\operatorname{min}}(\Sigma_{X_i})
%\geq\sigma^2>0.
%\]
%Then for every $\epsilon\in[0,1]$, it holds that
%\[
%\operatorname{dist}(S_{\hat{X}},S_X)\leq\frac{\epsilon}{\sigma}
%\quad
%\text{w.p.}
%\quad
%\geq1-4(d_1+d_2+2)\cdot e^{-\frac{n}{8\beta^2}\cdot f(\frac{\epsilon^2\sigma^2}{2})},
%\]
%where $f(z):=\min(z,z^2)$.
%\end{lemma}
%
%\begin{proof}
%Put $\delta:=\epsilon^2\sigma^2/2$, and consider the corresponding event in Lemma~\ref{lem.bound spectral distance}.
%Then for each $i\in\{1,2\}$, Weyl's inequality gives
%\[
%\lambda_{\operatorname{min}}(\Sigma_{\hat{X}_i})
%=\lambda_{\operatorname{min}}(\Sigma_{{X}_i}+\Sigma_{\hat{X}_i}-\Sigma_{{X}_i})
%\geq\lambda_{\operatorname{min}}(\Sigma_{{X}_i})-\|\Sigma_{\hat{X}_i}-\Sigma_{{X}_i}\|_{2\to2}
%\geq \sigma^2-\delta
%\geq \frac{\sigma^2}{2},
%\]
%where the last step uses the fact that $\epsilon\leq 1$.
%Lemma~\ref{lem.bound haus dist} then gives
%\[
%\operatorname{dist}(S_{\hat{X}},S_X)^2
%\leq\frac{\delta}{\sigma^2\cdot(\sigma^2/2)}
%=\frac{\epsilon^2}{\sigma^2}.
%\qedhere
%\]
%\end{proof}

In our case, Lemma~\ref{lem.key lemma}(ii) is immediate from Lemma~\ref{lem.lipschitz bound}.
For Lemma~\ref{lem.key lemma}(iii), our net-based argument requires a pointwise estimate:

\begin{lemma}
\label{lem.pointwise bound}
Suppose $\|X-\mu\|_{2,\infty}\leq\beta$ almost surely, and fix $A\in T_\alpha$.
Then for every $\delta\geq0$, it holds that
\[
|f_{\hat{X}}(A)-f_X(A)|\leq\delta
\quad
\text{w.p.}
\quad
\geq1-4(d_1+d_2)\cdot e^{-\frac{n}{2}\cdot f(\frac{\delta}{8\alpha^2\beta^2})},
\]
where $f(z):=\min(z,z^2)$.
\end{lemma}

\begin{proof}
First, add zero and expand the square to get
\begin{align*}
f_{\hat{X}}(A)
&=\mathbb{E}\|AD(\hat{X}-\hat\mu)\|_2^2
=\mathbb{E}\|AD(\hat{X}-\mu)-AD(\hat\mu-\mu)\|_2^2\\
&=\mathbb{E}\Big(\|AD(\hat{X}-\mu)\|_2^2-2\langle AD(\hat{X}-\mu),AD(\hat\mu-\mu)\rangle+\|AD(\hat\mu-\mu)\|_2^2\Big)\\
&=\mathbb{E}\|AD(\hat{X}-\mu)\|_2^2-\|AD(\hat\mu-\mu)\|_2^2.
\end{align*}
Next, put $Z_j:=X_j-\mu$.
Then the triangle inequality and Lemma~\ref{lem.v norm bound} together give
\begin{align*}
|f_{\hat{X}}(A)-f_X(A)|
&=\Big|\mathbb{E}\|AD(\hat{X}-\mu)\|_2^2-\|AD(\hat\mu-\mu)\|_2^2-\mathbb{E}\|AD(X-\mu)\|_2^2\Big|\\
&\leq\Big|\mathbb{E}\|AD(\hat{X}-\mu)\|_2^2-\mathbb{E}\|AD(X-\mu)\|_2^2\Big|+2\alpha^2\cdot\|\hat\mu-\mu\|_2^2\\
&=\bigg|\frac{1}{n}\sum_{j\in[n]}\Big(\|ADZ_j\|_2^2-\mathbb{E}\|ADZ\|_2^2\Big)\bigg|+2\alpha^2\cdot\bigg\|\frac{1}{n}\sum_{j\in[n]}Z_j\bigg\|_2^2.
\end{align*}
We will bound both terms above in a high-probability event by passing to (Vector) Hoeffding.
First, $0\leq\|ADZ_j\|_2^2\leq\|A\|_{2\to2}^2\|Z_j\|_2^2\leq 4\alpha^2\beta^2$ almost surely, and so 
\[
\Big|\|ADZ_j\|_2^2-\mathbb{E}\|ADZ\|_2^2\Big|
\leq4\alpha^2\beta^2
\]
almost surely.
As such, Hoeffding implies
\[
\bigg|\frac{1}{n}\sum_{j\in[n]}\Big(\|ADZ_j\|_2^2-\mathbb{E}\|ADZ\|_2^2\Big)\bigg|
\leq \delta_1
\quad
\text{w.p.}
\quad
\geq 1-2e^{-n\delta_1^2/(32\alpha^4\beta^4)}.
\]
Similarly, since $\|Z_j\|_2\leq \sqrt{2}\cdot\beta$ almost surely, Vector Hoeffding implies
\[
\bigg\|\frac{1}{n}\sum_{j\in[n]}Z_j\bigg\|_2
\leq \delta_2
\quad
\text{w.p.}
\quad
\geq 1-2(d_1+d_2+1)\cdot e^{-n\delta_2^2/(4\beta^2)}.
\]
The result then follows by setting $\delta_1=2\alpha^2\delta_2^2=\delta/2$ and applying the union bound.
\end{proof}

We are now ready to prove Theorem~\ref{thm.main result 1}.
What follows is a more explicit theorem statement.
(Note: We did not optimize the constants in this statement.)

\begin{theorem}
Suppose $\|X-\mu\|_{2,\infty}\leq\beta$ almost surely and $\min_{i\in\{1,2\}}\lambda_{\operatorname{min}}(\Sigma_{X_i})\geq\sigma^2>0$.
Fix any $\epsilon\in(0,2^5]$.
Then
\[
\displaystyle\Big|\min_{A\in S_{\hat{X}}}f_{\hat{X}}(A)-\min_{A\in S_{X}}f_X(A)\Big|\leq \epsilon\cdot\frac{\beta^2}{\sigma^2}
\]
in an event of probability $\geq 1-p$, provided
\[
n\geq\max\bigg\{
\tfrac{2^{15}}{\epsilon^2}\Big(k(d_1+d_2)\log(\tfrac{2^{22}k}{\epsilon^2})+\log(\tfrac{2}{p})\Big),
\tfrac{2^{25}}{\epsilon^4}(\tfrac{\beta}{\sigma})^4\Big(\log(2^3(d_1+d_2))+\log(\tfrac{2}{p})\Big)
\bigg\}.
\]
\end{theorem}

\begin{proof}
Let $N_{\alpha,\eta}$ denote the net described in Lemma~\ref{lem.net exists}, and let $\mathcal{E}_{\delta,\alpha,\eta,\gamma}$ denote the event
\[
\Big\{
~~
\max_{i\in\{1,2\}}\|\Sigma_{\hat{X}_i}-\Sigma_{X_i}\|_{2\to2}\leq\delta
\quad
\text{and}
\quad
\max_{A\in N_{\alpha,\eta}}|f_{\hat{X}}(A)-f_X(A)|+16(\alpha+\eta)\beta^2\eta\leq\gamma
~~
\Big\}.
\]
Let $\xi\in[0,1]$ be arbitrary (to be selected later), and put $\delta:=\xi^2\sigma^2/2$ and $\alpha:=2/\sigma$.
Then the first part of $\mathcal{E}_{\delta,\alpha,\eta,\gamma}$ together with Weyl's inequality gives
\[
\lambda_{\operatorname{min}}(\Sigma_{\hat{X}_i})
=\lambda_{\operatorname{min}}(\Sigma_{{X}_i}+\Sigma_{\hat{X}_i}-\Sigma_{{X}_i})
\geq\lambda_{\operatorname{min}}(\Sigma_{{X}_i})-\|\Sigma_{\hat{X}_i}-\Sigma_{{X}_i}\|_{2\to2}
\geq \sigma^2-\delta
\geq \frac{\sigma^2}{2}
\]
for each $i\in\{1,2\}$, where the last step uses the fact that $\xi\leq 1$.
Lemma~\ref{lem.bound haus dist} then gives
\begin{itemize}
\item[(i)]
$\operatorname{dist}(S_{\hat{X}},S_X)
\leq(\frac{\delta}{\sigma^2\cdot(\sigma^2/2)})^{1/2}
=\frac{\xi}{\sigma}$.
\end{itemize}
In addition, by Lemma~\ref{lem.SX is bounded}, every $A\in S_{\hat{X}}\cup S_X$ satisfies $\|A\|_V\leq\sqrt{2}/\sigma\leq\alpha$, and so we have $S_{\hat{X}}\cup S_X\subseteq T_\alpha$.
Lemma~\ref{lem.lipschitz bound} then implies 
\begin{itemize}
\item[(ii)]
$f_{\hat{X}},f_X\colon(S_{\hat{X}}\cup S_X,\|\cdot\|_V)\to\mathbb{R}$ are both $8\alpha\beta^2$-Lipschitz.
\end{itemize}
Taking $f(A):=|f_{\hat{X}}(A)-f_X(A)|$, then Lemma~\ref{lem.lipschitz bound} also implies that $f\colon(T_{\alpha+\eta},\|\cdot\|_V)\to\mathbb{R}$ is $16(\alpha+\eta)\beta^2$-Lipschitz.
This together with the second part of $\mathcal{E}_{\delta,\alpha,\eta,\gamma}$ then gives
\begin{itemize}
\item[(iii)]
$|f_{\hat{X}}(A)-f_X(A)|\leq\gamma$ for every $A\in S_X\cup S_Y$.
\end{itemize}
Now that we have (i)--(iii), we may conclude by Lemma~\ref{lem.key lemma} that
\[
\Big|\min_{A\in S_{\hat{X}}}f_{\hat{X}}(A)-\min_{A\in S_{X}}f_X(A)\Big|
\leq \frac{16\xi\beta^2}{\sigma^2}+\gamma
\]
over the event $\mathcal{E}_{\delta,\alpha,\eta,\gamma}$.
At this point, we select $\xi:=2^{-5}\epsilon$ so that $\delta=2^{-11}\epsilon^2\sigma^2$, and we select $\eta:=2^{-8}\epsilon\sigma^{-1}$ and $\gamma:=2^{-1}\epsilon\beta^2\sigma^{-2}$ so that the right-hand size above equals $\epsilon\beta^2\sigma^{-2}$.
Then since $\epsilon\leq 2^5$ and $\beta\geq\sigma$, the union bound together with Lemmas~\ref{lem.bound spectral distance}, \ref{lem.net exists}, and~\ref{lem.pointwise bound} gives
\begin{align*}
&\mathbb{P}[(\mathcal{E}_{\delta,\alpha,\eta,\gamma})^c]\\
&\leq 8(d_1+d_2)\cdot e^{-\frac{n}{2}\cdot (2^{-12}\epsilon^2\sigma^2\beta^{-2})^2}+(1+2^{10}\sqrt{2k}\epsilon^{-1})^{k(d_1+d_2)}\cdot 4(d_1+d_2)\cdot e^{-\frac{n}{2}\cdot (2^{-7}\epsilon)^2}\\
&\leq \operatorname{exp}\Big[\log(2^3(d_1+d_2))-n\cdot 2^{-25}(\epsilon\sigma\beta^{-1})^4\Big]+\operatorname{exp}\Big[k(d_1+d_2)\log(2^{22}k\epsilon^{-2})-n\cdot 2^{-15}\epsilon^2\Big],
\end{align*}
and each term of the final sum is smaller than $p/2$ by our choice of $n$.
\end{proof}

\section{Proof of Theorem~\ref{thm.main result}}
\label{sec.proof of main result 2}

%\subsection{Discrete random vectors}
%
%Suppose $Z_i$ is a random function of $X_i$.
%Since $H(Y)=H(Y|Z_1)+I(Y;Z)$ and $H(Z_1)=H(Z_1|Z_2)+I(Z_1;Z_2)$, combining these gives
%\[
%\text{error}
%:=H(Y|Z_1)+H(Z_1|Z_2)
%=H(Y)+H(Z_1|Y)-I(Z_1;Z_2)
%=H(Y,Z_1)-I(Z_1;Z_2).
%\]
%The data processing inequality then gives
%\[
%\text{error}
%\geq H(Y)-I(X_1;X_2).
%\]
%This means we are forced to have error if there isn't enough mutual information available.
%Note that this bound isn't aways achievable (e.g., if $Y$ is deterministic but $X_1=X_2$ is not).
%
%\textcolor{blue}{what is the minimum error for a given $(X_1,X_2,Y)$?}
%
%\textcolor{blue}{how would we minimize error in theory?}

%\subsection{Affine linear random vectors}

The following lemma will help us prove both parts of the result:

\begin{lemma}
\label{lem.key lemma 2}
Suppose $A_1,A_2,S_1,S_2$ are real matrices such that
\[
A_1S_1
=A_2S_2
\qquad
\text{and}
\qquad
\operatorname{im}A_i^\top
\subseteq\operatorname{im}S_i,
\qquad
i\in\{1,2\}.
\]
Then $\operatorname{ker}A_iS_i=\operatorname{ker}S_1+\operatorname{ker}S_2$ if and only if
\begin{equation}
\label{eq.key lemma 2 condition}
\operatorname{ker}[A_1,-A_2]
=\operatorname{im}[S_1;S_2]+(\operatorname{ker}S_1^\top\oplus\operatorname{ker}S_2^\top).
\end{equation}
%$\operatorname{ker}[A_1,-A_2]=\operatorname{im}[S_1;S_2]$ implies .
%Furthermore, the converse holds when $S_1$ and $S_2$ both have full row rank.
\end{lemma}

%The converse requires an additional condition (such as the full row rank condition) since otherwise there exists a counterexample:
%Take $S_1=S_2=[1;0]$ and $A_1=A_2=[1,1]$.
%Then $A_1S_1=1=A_2S_2=:T$ and $\operatorname{ker}T=\{0\}=\operatorname{ker}S_1+\operatorname{ker}S_2$, but $[A_1,-A_2]$ has a $3$-dimensional kernel, while $[S_1;S_2]$ has a $1$-dimensional image.

\begin{proof}
Let $d_i$ and $r_i$ denote the number of rows and the rank of $S_i$, respectively.
Let $V_i$ denote a $d_i\times r_i$ matrix whose columns form an orthonormal basis for $\operatorname{im}S_i$.
We first claim that \eqref{eq.key lemma 2 condition} holds if and only if $\operatorname{ker}[A_1V_1,-A_2V_2]=\operatorname{im}[V_1^\top S_1;V_2^\top S_2]$.
To see ($\Rightarrow$), note that
\begin{align*}
\operatorname{ker}[A_1V_1,-A_2V_2]
&=\left[\begin{array}{cc}V_1^\top&0\\0&V_2^\top\end{array}\right]\Big(\operatorname{ker}[A_1,-A_2]\cap(\operatorname{im}S_1\oplus\operatorname{im}S_2)\Big)\\
&=\left[\begin{array}{cc}V_1^\top&0\\0&V_2^\top\end{array}\right]\operatorname{im}[S_1;S_2]\\
&=\operatorname{im}[V_1^\top S_1;V_2^\top S_2].
\end{align*}
For ($\Leftarrow$), observe that since $\operatorname{im}A_i^\top\subseteq\operatorname{im}S_i$, it holds that $A_i=A_iV_iV_i^\top$, and so
\begin{align*}
\operatorname{ker}[A_1,-A_2]
&=\operatorname{ker}[A_1V_1V_1^\top,-A_2V_2V_2^\top]\\
&=\left[\begin{array}{cc}V_1&0\\0&V_2\end{array}\right]\operatorname{ker}[A_1V_1,-A_2V_2]+\Big((\operatorname{im}V_1)^\perp\oplus(\operatorname{im}V_2)^\perp\Big)\\
&=\left[\begin{array}{cc}V_1&0\\0&V_2\end{array}\right]\operatorname{im}[V_1^\top S_1;V_2^\top S_2]+\Big((\operatorname{im}V_1)^\perp\oplus(\operatorname{im}V_2)^\perp\Big)\\
&=\operatorname{im}[S_1;S_2]+(\operatorname{ker}S_1^\top\oplus\operatorname{ker}S_2^\top).
\end{align*}
In addition, $A_iS_i=A_iV_iV_i^\top S_i$.
Overall, if $\operatorname{im}S_i$ is a proper subspace of $\mathbb{R}^{d_i}$, then we may redefine $S_i\leftarrow V_i^\top S_i$ without loss of generality.
As such, from now on, we assume that $A_1S_1=A_2S_2=:T$ and $\operatorname{im}S_i=\mathbb{R}^{d_i}$ for both $i\in\{1,2\}$, and our task is to prove the equivalence
\[
\operatorname{ker}T=\operatorname{ker}S_1+\operatorname{ker}S_2
\quad
\Longleftrightarrow
\quad
\operatorname{ker}[A_1,-A_2]
=\operatorname{im}[S_1;S_2].
\]

($\Leftarrow$)
By Lemma~\ref{lem.kernel limit}, it suffices to show $\operatorname{ker}T\subseteq\operatorname{ker}S_1+\operatorname{ker}S_2$.
Suppose $x\in\operatorname{ker}T$.
Then $A_iS_ix=0$, and so $[\pm S_1x;S_2x]\in\operatorname{ker}[A_1,-A_2]$, which by averaging gives $[0;S_2x]\in\operatorname{ker}[A_1,-A_2]$.
Since $\operatorname{ker}[A_1,-A_2]=\operatorname{im}[S_1;S_2]$ by assumption, there must exist $v$ such that $S_1v=0$ and $S_2v=S_2x$, that is, $x\in v+\operatorname{ker}S_2\subseteq\operatorname{ker}S_1+\operatorname{ker}S_2$, as desired.

($\Rightarrow$)
Since $A_1S_2=A_2S_2$ by assumption, it holds that $\operatorname{ker}[A_1,-A_2]\supseteq\operatorname{im}[S_1;S_2]$.
It therefore suffices to prove $\operatorname{dim}\operatorname{ker}[A_1,-A_2]\leq\operatorname{rank}[S_1;S_2]$.
To do so, we will apply the following intermediate claims:
\begin{itemize}
\item[(i)]
$\operatorname{ker}A_1=S_1\operatorname{ker}S_2$.
\item[(ii)]
$\operatorname{dim}S_1\operatorname{ker}S_2
=\operatorname{dim}\operatorname{ker}S_2-\operatorname{dim}\operatorname{ker}[S_1;S_2]$.
\end{itemize}
First, we verify (i).
For ($\subseteq$), select $x\in\operatorname{ker}A_1$.
Since $S_1$ has full row rank by assumption, there exists $y$ such that $x=S_1y$.
It follows that $y\in\operatorname{ker}T$.
By assumption, we may decompose $y=u_1+u_2$ with $u_i\in\operatorname{ker}S_i$.
Then $x=S_1(u_1+u_2)=S_1u_2\in S_1\operatorname{ker}S_2$.
For ($\supseteq$), select $u_2\in\operatorname{ker}S_2\subseteq\operatorname{ker}T$.
Then $0=Tu_2=A_1S_1u_2$, and so $S_1u_2\in\operatorname{ker}A_1$.
For (ii), select a basis $B_0$ for $\operatorname{ker}[S_1;S_2]=\operatorname{ker}S_1\cap\operatorname{ker}S_2$ and extend to a basis $B_2$ for $\operatorname{ker}S_2$.
Then $\operatorname{span}\{S_1x\}_{x\in B_2}=S_1\operatorname{ker}S_2$.
Since $S_1x=0$ for every $x\in B_0$, we have $\operatorname{span}\{S_1x\}_{x\in B_2\setminus B_0}=S_1\operatorname{ker}S_2$.
By construction, no nontrivial linear combination of $B_2\setminus B_0$ resides in $\operatorname{ker}S_1$, and so $\{S_1x\}_{x\in B_2\setminus B_0}$ is linearly independent.
It follows that $\{S_1x\}_{x\in B_2\setminus B_0}$ is a basis for $S_1\operatorname{ker}S_2$, and the claim follows by counting.

At this point, it is convenient to enunciate dimensions: $A_i\in\mathbb{R}^{k\times d_i}$ and $S_i\in\mathbb{R}^{d_i\times D}$.
In what follows, we obtain the result after multiple applications of the rank--nullity theorem.
First, we apply rank--nullity on $[A_1,-A_2]$ and on $A_1$ to get
\begin{align*}
\operatorname{dim}\operatorname{ker}[A_1,-A_2]
=d_1+d_2-\operatorname{rank}[A_1,-A_2]
\leq d_1+d_2-\operatorname{rank}A_1
=d_2+\operatorname{dim}\operatorname{ker}A_1.
\end{align*}
Next, we apply (i) and (ii) and the fact that $S_2$ has full row rank to get
\begin{align*}
\operatorname{dim}\operatorname{ker}[A_1,-A_2]
\leq d_2+\operatorname{dim}\operatorname{ker}A_1
&=d_2+\operatorname{dim}\operatorname{ker}S_2-\operatorname{dim}\operatorname{ker}[S_1;S_2]\\
&=\operatorname{rank}S_2+\operatorname{dim}\operatorname{ker}S_2-\operatorname{dim}\operatorname{ker}[S_1;S_2].
\end{align*}
Finally, we apply rank--nullity on $S_2$ and on $[S_1;S_2]$ to get
\begin{align*}
\operatorname{dim}\operatorname{ker}[A_1,-A_2]
\leq\operatorname{rank}S_2+\operatorname{dim}\operatorname{ker}S_2-\operatorname{dim}\operatorname{ker}[S_1;S_2]
&=D-\operatorname{dim}\operatorname{ker}[S_1;S_2]\\
&=\operatorname{rank}[S_1;S_2].
\qedhere
\end{align*}
\end{proof}

\begin{lemma}
\label{lem.generic}
Fix any $m\times n$ matrix $A$ of rank $r$.
Then $AX$ also has rank $r$ for a generic $n\times p$ matrix $X$ that satisfies $X1=0$, provided $p\geq r+1$.
\end{lemma}

\begin{proof}
First, we write $X=[x_{ij}]_{i\in[n],j\in[p]}$.
Since $X1=0$, we observe that $X$ consists of $n(p-1)$ free variables $\{x_{ij}\}_{i\in[n],j\in[p-1]}$ that together determine the final column $x_{ip}=-\sum_{j\in[p-1]}x_{ij}$.
Select size-$r$ index sets $S\subseteq[m]$ and $T\subseteq[n]$ such that the $r\times r$ submatrix $A_{ST}$ of $A$ has rank $r$.
Let $A_S$ denote $r\times n$ submatrix of $A$ whose row indices reside in $S$, and let $X_r$ denote the $n\times r$ submatrix of $X$ whose column indices reside in $[r]$.
Then $p(X):=\det(A_SX_r)$ is a polynomial in $\{x_{ij}\}_{i\in[n],j\in[p-1]}$ that we claim is nonzero.
To see this, write $T=\{t_1,\ldots,t_r\}$ and consider the $n\times p$ matrix $B$ defined by
\[
B_{ij}
=\left\{\begin{array}{rl}1&\text{if }i=t_j\\-1&\text{if }i\in T, j=p\\0&\text{otherwise.}\end{array}\right.
\]
Then $B1=0$ and $A_SB_r=A_{ST}$, meaning $p(B)=\det(A_SB_r)=\det(A_{ST})\neq0$.
This establishes that $p(X)$ is a nonzero polynomial, and so the complement of its zero set is generic.
Over this generic set of $X$'s, since $A_SX_r$ is a submatrix of $AX$, it holds that
\[
r
=\operatorname{rank}A_SX_r
\leq\operatorname{rank}AX
\leq\operatorname{rank}A
=r.
\qedhere
\]
\end{proof}

\begin{proof}[Proof of Theorem~\ref{thm.main result}(a)]
For the requisite function $\mathcal{D}$, we run matching component analysis (MCA, Algorithm~\ref{alg.mca}) with a data-dependent choice for $k$, namely,
\[
k:=\operatorname{dim}(\operatorname{im}Z_1^\top \cap \operatorname{im}Z_2^\top).
\]
Here, $Z_1$ and $Z_2$ are determined in the normalization stage of MCA.
Notice that MCA requires $k\geq1$.
As such, in the degenerate case where $k=0$, we say $\mathcal{D}$ outputs $A_i=0\in\mathbb{R}^{1\times d_i}$ and $b_i=0\in\mathbb{R}$.

We claim that $\mathcal{D}$ witnesses that $\operatorname{ALM}(d_1,d_2,n)$ is exactly matchable.
To see this, fix $D\in\mathbb{N}$, select any continuous distribution $\mathbb{P}$ over $\mathbb{R}^D$, select $S_i\in\mathbb{R}^{d_i\times D}$ and $\mu_i\in\mathbb{R}^{d_i}$ for $i\in\{1,2\}$, and then draw $\{\omega_j\}_{j\in[n]}$ independently with distribution $\mathbb{P}$.
We run MCA on data of the form $x_{ij}:=S_i\omega_j+\mu_i$ for $i\in\{1,2\}$ and $j\in[n]$.
Put $\overline{\omega}:=\frac{1}{n}\sum_{j\in[n]}\omega_j$.
Then $\overline{x}_i=S_i\overline{\omega}+\mu_i$, and so $x_{ij}-\overline{x}_i=S_i(\omega_j-\overline{\omega})$.
Let $F$ denote the $D\times n$ matrix whose $j$th column is $\omega_j-\overline{\omega}$.
Then $Z_i=\Lambda_i^{-1/2}V_i^\top S_iF$.
The choice of $\Lambda_i$ and $V_i$ ensures that the columns of $\frac{1}{\sqrt{n}}Z_i^\top$ are orthonormal.
As such, the singular values of $\frac{1}{n}Z_1Z_2^\top$ are cosines of the principal angles between $\operatorname{im}Z_1^\top$ and $\operatorname{im}Z_2^\top$.
It follows that $\|Z_1Z_2^\top\|_{2\to2}\leq n$, and the multiplicity of the singular value $n$ equals our choice for $k$.

\textbf{Case I:} $k\geq1$.
MCA finds $W_i\in\mathbb{R}^{k\times r_i}$ with orthonormal columns for $i\in\{1,2\}$ such that $nW_1W_2^\top=Z_1Z_2^\top$.
This in turn implies that $nI_k=W_1^\top Z_1Z_2^\top W_2$, and since the columns of $\frac{1}{\sqrt{n}}Z_i^\top W_i$ are orthonormal, it follows that $W_1^\top Z_1=W_2^\top Z_2$.
Since $A_i:=W_i^\top\Lambda_i^{-1/2} V_i^\top$, this then implies
\[
A_1S_1F
=W_1^\top\Lambda_1^{-1/2} V_1^\top S_1F
=W_1^\top Z_1
=W_2^\top Z_2
=W_2^\top\Lambda_2^{-1/2} V_2^\top S_2F
=A_2S_2F.
\]
Equivalently, we have $[A_1,-A_2][S_1;S_2]F=0$.
Next, since $n\geq d_1+d_2+1\geq \operatorname{rank}[S_1;S_2]+1$, Lemma~\ref{lem.generic} implies that the following holds almost surely:
\[
\operatorname{im}[S_1;S_2]
=\operatorname{im}[S_1;S_2]F
\subseteq\operatorname{ker}[A_1,-A_2].
\]
As such, $[A_1,-A_2][S_1;S_2]=0$, that is, $A_1S_1=A_2S_2$.
Considering $b_i=-A_i(S_i\overline{\omega}+\mu_i)$, we further have
\[
A_1(S_1\omega+\mu_1)+b_1
=A_1S_1(\omega-\overline{\omega})
=A_2S_2(\omega-\overline{\omega})
=A_2(S_2\omega+\mu_2)+b_2
\]
for every $\omega\in\mathbb{R}^D$.
This establishes Definition~\ref{eq.alm}(i).
For Definition~\ref{eq.alm}(ii), first note that
\begin{equation}
\label{eq.subspace containment}
\operatorname{im}A_i^\top
=\operatorname{im}V_i\Lambda_i^{-1/2}W_i
\subseteq\operatorname{im}V_i
\subseteq\operatorname{im}S_i
\end{equation}
for both $i\in\{1,2\}$, and so the hypothesis of Lemma~\ref{lem.key lemma 2} is satisfied.
Taking orthogonal complements of \eqref{eq.subspace containment} gives $\operatorname{ker}S_i^\top\subseteq\operatorname{ker}A_i$.
Since $\operatorname{ker}[A_1,-A_2]$ is closed under addition, this then implies
\begin{equation}
\label{eq.subspace containment before dim count}
\operatorname{ker}[A_1,-A_2]
\supseteq\operatorname{im}[S_1;S_2]+(\operatorname{ker}S_1^\top\oplus\operatorname{ker}S_2^\top).
\end{equation}
We count dimensions to demonstrate equality.
For the left-hand side, the rank--nullity theorem gives
\[
\operatorname{dim}\operatorname{ker}[A_1,-A_2]
=d_1+d_2-\operatorname{rank}[A_1,-A_2]
=d_1+d_2-k.
\]
For the right-hand side, notice that $\operatorname{im}[S_1;S_2]$, $\operatorname{ker}S_1^\top\oplus\{0\in\mathbb{R}^{d_2}\}$, and $\{0\in\mathbb{R}^{d_1}\}\oplus\operatorname{ker}S_2^\top$ are pairwise orthogonal, and so
\[
\operatorname{dim}\Big(\operatorname{im}[S_1;S_2]+(\operatorname{ker}S_1^\top\oplus\operatorname{ker}S_2^\top)\Big)
=\operatorname{rank}[S_1;S_2]+\operatorname{dim}\operatorname{ker}S_1^\top+\operatorname{dim}\operatorname{ker}S_2^\top.
\]
Put $r_i:=\operatorname{rank}S_i=\operatorname{rank}S_iF=\operatorname{rank}Z_i$, where the second equality holds almost surely by Lemma~\ref{lem.generic}.
Then
\begin{align*}
\operatorname{rank}[S_1;S_2]
=\operatorname{rank}[S_1;S_2]F
&=\operatorname{rank}\left[\begin{array}{cc}\Lambda_1^{-1/2}V_1^\top&0\\0&\Lambda_2^{-1/2}V_2^\top\end{array}\right]\left[\begin{array}{c}S_1F\\S_2F\end{array}\right]\\
&=\operatorname{rank}[Z_1;Z_2]\\
&=\operatorname{rank}[Z_1^\top,Z_2^\top]
=\operatorname{dim}(\operatorname{im}Z_1^\top+\operatorname{im}Z_2^\top)
=r_1+r_2-k.
\end{align*}
Also, $\operatorname{dim}\operatorname{ker}S_i^\top=d_i-r_i$ for both $i\in\{1,2\}$ by rank--nullity.
Overall, \eqref{eq.key lemma 2 condition} holds, and so we may conclude Definition~\ref{eq.alm}(ii).

\textbf{Case II:} $k=0$.
Definition~\ref{eq.alm}(i) holds since both sides of the equality are zero.
For Definition~\ref{eq.alm}(ii), we again appeal to Lemma~\ref{lem.key lemma 2}.
In this case, \eqref{eq.subspace containment before dim count} is immediate since $\operatorname{ker}[A_1,-A_2]=\mathbb{R}^{d_1+d_2}$, and equality follows from the same dimension count.
\end{proof}

\begin{proof}[Proof of Theorem~\ref{thm.main result}(b)]
Suppose to the contrary that $\operatorname{ALM}(d_1,d_2,n)$ is exactly matchable for some $n<d_1+d_2+1$ with witness $\mathcal{D}$.
We may take $n=d_1+d_2$ without loss of generality.
Put $D=d_1+d_2$ and let $\mathbb{P}_1$ be any continuous distribution that is supported on all of $\mathbb{R}^D$.
Let $\{\omega_j\}_{j\in[D]}$ denote independent random variables with distribution $\mathbb{P}_1$, and let $\mathcal{V}$ denote the distribution of the shortest vector $v(\{\omega_j\}_{j\in[D]})$ in the affine hull of $\{\omega_j\}_{j\in[D]}$.
Notice that $v(\{\omega_j\}_{j\in[D]})\neq0$ almost surely, and for every $v\in\mathbb{R}^D\setminus\{0\}$, $v(\{\omega_j\}_{j\in[D]})=v$ is equivalent to having $\omega_j\in v^\perp+v$ for every $j\in[D]$.
As such, $\{\omega_j\}_{j\in[D]}$ remain independent after conditioning on $v(\{\omega_j\}_{j\in[D]})$.
Let $\mathbb{P}_1|_v$ denote the distribution of $\omega_j$ conditioned on $\omega_j\in v^\perp+v$.
Select any piecewise continuous mapping that sends $v\in\mathbb{R}^D\setminus\{0\}$ to a $D\times (D-1)$ matrix $S^{(v)}$ whose columns form an orthonormal basis for $v^\perp$, and define $\mathbb{P}_2^{(v)}$ to be the continuous distribution on $\mathbb{R}^{D-1}$ such that if $X$ has distribution $\mathbb{P}_2^{(v)}$ then $S^{(v)}X$ has distribution $\mathbb{P}_1|_v$.
Now put $[S_1;S_2]:=I_D$, $[S^{(v)}_1;S^{(v)}_2]:=S^{(v)}$, and $[v_1;v_2]:=v$.
Drawing $V\sim\mathcal{V}$, we therefore have
\begin{equation}
\label{eq.equal dist}
\mathcal{E}_{\mathbb{P}_1}(S_1,0,S_2,0)
\equiv \mathcal{E}_{\mathbb{P}_2^{(V)}}(S^{(V)}_1,V_1,S^{(V)}_2,V_2).
\end{equation}
Here, $\equiv$ denotes equality in distribution.
At this point, we define
\begin{align*}
(A_1,b_1,A_2,b_2)
&:=(\mathcal{D}\circ\mathcal{E}_{\mathbb{P}_1})(S_1,0,S_2,0),\\
(A_1^{(v)},b_1^{(v)},A_2^{(v)},b_2^{(v)})
&:=(\mathcal{D}\circ\mathcal{E}_{\mathbb{P}_2^{(v)}})(S^{(v)}_1,v_1,S^{(v)}_2,v_2),
\qquad v\in\mathbb{R}^D\setminus\{0\}.
\end{align*}
By assumption, we have both
\begin{itemize}
\item[(i)]
$A_1(S_1\omega+0)+b_1=A_2(S_2\omega+0)+b_2$ for all $\omega\in\mathbb{R}^D$, and
\item[(ii)]
$\operatorname{ker}A_iS_i=\operatorname{ker}S_1+\operatorname{ker}S_2$.
\end{itemize}
Setting $\omega=0$ in (i) reveals that $b_1=b_2$, which implies that $A_1S_1\omega=A_2S_2\omega$ for all $\omega\in\mathbb{R}^D$, i.e., $A_1S_1=A_2S_2$.
Also, our choice of $S_i$ ensures that $\operatorname{im}A_i^\top\subseteq\mathbb{R}^{d_i}=\operatorname{im}S_i$ for both $i\in\{1,2\}$, and so the hypothesis of Lemma~\ref{lem.key lemma 2} is satisfied.
As such, (ii) and Lemma~\ref{lem.key lemma 2} together imply that
\[
\operatorname{ker}[A_1,-A_2]
=\operatorname{im}[S_1;S_2]+(\operatorname{ker}S_1^\top\oplus\operatorname{ker}S_2^\top)
=\operatorname{im}[S_1;S_2]
=\operatorname{im}S.
\]
The same argument gives $\operatorname{ker}[A_1^{(v)},-A_2^{(v)}]=\operatorname{im}S^{(v)}$ for generic $v\neq0$.
Now define the function $\mathcal{K}\colon(X_1,y_1,X_2,y_2)\mapsto\operatorname{dim}\operatorname{ker}[X_1,-X_2]$.
Then continuing \eqref{eq.equal dist}, we have
\begin{align*}
D
=\operatorname{rank}S
=\operatorname{dim}\operatorname{ker}[A_1,-A_2]
&=\mathcal{K}(A_1,b_1,A_2,b_2)\\
&\equiv\mathcal{K}(A_1^{(V)},b_1^{(V)},A_2^{(V)},b_2^{(V)})\\
&=\operatorname{dim}\operatorname{ker}[A_1^{(V)},-A_2^{(V)}]
=\operatorname{rank}S^{(V)}
=D-1
\end{align*}
almost surely, a contradiction.
\end{proof}

\section*{Acknowledgments}

CC and DGM were partially supported by the Air Force Summer Faculty Fellowship Program. 
DGM was also supported by AFOSR FA9550-18-1-0107, NSF DMS 1829955 and the Simons Institute of the Theory of Computing.
TS was supported in part by AFOSR LRIR 18RYCOR011.


\begin{thebibliography}{WW}

\bibitem{Bernhardsson:online}
E.\ Bernhardsson,
Analyzing 50k fonts using deep neural networks,
\url{https://erikbern.com}

\bibitem{Bhatia:13}
R.\ Bhatia, 
Matrix analysis,
Springer Science \& Business Media, 2013.

\bibitem{deng2009imagenet}
J.\ Deng, W.\ Dong, R.\ Socher, L.-J.\ Li, K.\ Li, L.\ Fei-Fei,
Imagenet:\ A large-scale hierarchical image database,
CVPR 2009, 248--255.

\bibitem{higham2018deep}
C.\ F.\ Higham D.\ J.\ Higham,
Deep learning:\ An introduction for applied mathematicians,
arXiv:1801.05894

\bibitem{HornJ:12}
R.\ A.\ Horn, C.\ R.\ Johnson,
Matrix Analysis,
Cambridge University Press, 2012.

\bibitem{irving1999classification}
W.\ W.\ Irving, G.\ J.\ Ettinger,
Classification of targets in synthetic aperture radar imagery via quantized grayscale matching,
Proc.\ SPIE 3721 (1999) 320--331.

\bibitem{isola2017image}
P.\ Isola, J.-Y.\ Zhu, T.\ Zhou, A.\ A.\ Efros,
Image-to-image translation with conditional adversarial networks,
CVPR 2017, 1125--1134.

\bibitem{krizhevsky2010convolutional}
A.\ Krizhevsky,
Convolutional Seep Belief Networks on CIFAR-10,
\url{https://www.cs.toronto.edu/~kriz/conv-cifar10-aug2010.pdf}

\bibitem{krizhevsky2012imagenet}
A.\ Krizhevsky, I.\ Sutskever, G.\ E.\ Hinton,
ImageNet classification with deep convolutional neural networks,
NIPS 2012, 1097--1105.

\bibitem{LeCunCB:online}
Y.\ LeCun, C.\ Cortes, C.\ J.\ C.\ Burges, 
The MNIST Database of handwritten digits,
\url{http://yann.lecun.com/exdb/mnist/}

\bibitem{lee2018diverse}
H.-Y.\ Lee, H.-Y.\ Tseng, J.-B.\ Huang, M.\ Singh, M.-H.\ Yang,
Diverse image-to-image translation via disentangled representations,
ECCV 2018, 35--51.

\bibitem{lewis2018generative}
B.\ Lewis, J.\ Liu, A.\ Wong,
Generative adversarial networks for SAR image realism,
Proc.\ SPIE 10647 (2018) 1064709.

\bibitem{LewisEtal:19}
B.\ Lewis, T.\ Scarnati, E.\ Sudkamp, J.\ Nehrbass, S.\ Rosencrantz, E.\ Zelnio,
A SAR dataset for ATR Development:\ The Synthetic and Measured Paired Labeled Experiment (SAMPLE),
Proc.\ SPIE 10987 (2019) 109870H.

\bibitem{MackeyEtal:12}
L.\ Mackey, M.\ I.\ Jordan, R.\ Y.\ Chen, B.\ Farrell, J.\ A.\ Tropp,
Matrix Concentration Inequalities via the Method of Exchangeable Pairs,
Ann.\ Probab.\ 42 (2012) 906--945.

\bibitem{McWhirterMV:18}
C.\ McWhirter, D.\ G.\ Mixon, S.\ Villar,
SqueezeFit:\ Label-aware dimensionality reduction by semidefinite programming,
arXiv:1812.02768

\bibitem{motiian2017few}
S.\ Motiian, Q.\ Jones, S.\ Iranmanesh, G.\ Doretto,
Few-shot adversarial domain adaptation,
NIPS 2017, 6670--6680.
  
\bibitem{paulson2018synthetic}
C.\ Paulson, J.\ Wilson, T.\ Lewis,
Synthetic aperture radar quantized grayscale reference automatic target recognition algorithm,
Proc.\ SPIE 10647 (2018) 106470P.

\bibitem{recht2018cifar}
B.\ Recht, R.\ Roelofs, L.\ Schmidt, V.\ Shankar,
Do CIFAR-10 classifiers generalize to CIFAR-10?,
arXiv:1806.00451

\bibitem{recht2019imagenet}
B.\ Recht, R.\ Roelofs, L.\ Schmidt, V.\ Shankar,
Do ImageNet classifiers generalize to ImageNet?,
arXiv:1902.10811

\bibitem{scarnati2019deep}
T.\ Scarnati, B.\ Lewis,
A deep learning approach to the synthetic and measured paired and labeled experiment (SAMPLE) challenge problem,
Proc.\ SPIE 10987 (2019) 109870G.

\bibitem{tosic2011dictionary}
I.\ Tosic, P.\ Frossard,
Dictionary learning:\ What is the right representation for my signal?,
IEEE Signal Process.\ Mag.\ 28 (2011) 27--38.

\bibitem{Tropp:15}
J.\ A.\ Tropp,
An introduction to matrix concentration inequalities,
Foundations and Trends in Machine Learning 8 (2015) 1--230.

\bibitem{tzeng2017adversarial}
E.\ Tzeng, J.\ Hoffman, K.\ Saenko, T.\ Darrell,
Adversarial discriminative domain adaptation,
CVPR 2017, 7167--7176.

\bibitem{Vershynin:18}
R.\ Vershynin,
High-dimensional probability:\ An introduction with applications in data science,
Cambridge University Press, 2018.

\bibitem{VillarBBW:16}
S.\ Villar, A.\ S.\ Bandeira, A.\ J.\ Blumberg, R.\ Ward,
A polynomial-time relaxation of the Gromov--Hausdorff distance,
arXiv:1610.05214

\bibitem{yang2014sparse}
M.\ Yang, L.\ Zhang, X.\ Feng, D.\ Zhang,
Sparse representation based Fisher discrimination dictionary learning for image classification,
Int.\ J.\ Comput.\ Vis.\ 109 (2014) 209--232.

\end{thebibliography}
\end{document}